\documentclass[reqno]{amsart}

\usepackage{a4wide}
\usepackage{color}
\usepackage{mathrsfs}
\usepackage{mathtools}
\usepackage{tikz-cd}
\usepackage{amsmath}
\usepackage{amssymb}
\numberwithin{equation}{section}
\usepackage[colorlinks,citecolor=green,linkcolor=red]{hyperref}
\usepackage{hyphenat}
\usepackage[latin1]{inputenc}

\newcommand{\N}{\mathbb{N}}
\newcommand{\R}{\mathbb{R}}
\newcommand{\sfd}{{\sf d}}
\renewcommand{\d}{{\mathrm d}}
\newcommand{\restr}[1]{\lower3pt\hbox{\(|_{#1}\)}}

\newcommand{\nchi}{{\raise.3ex\hbox{\(\chi\)}}}
\newcommand{\fr}{\penalty-20\null\hfill\(\blacksquare\)}
\newcommand{\X}{{\rm X}}
\newcommand{\Y}{{\rm Y}}
\newcommand{\mm}{\mathfrak m}

\newtheorem{theorem}{Theorem}[section]
\newtheorem{corollary}[theorem]{Corollary}

\newtheorem{proposition}[theorem]{Proposition}
\newtheorem{definition}[theorem]{Definition}
\newtheorem{example}[theorem]{Example}
\newtheorem{remark}[theorem]{Remark}

\title{Limits and colimits in the category of Banach $L^0$-modules}
\author{Enrico Pasqualetto}
\address{Department of Mathematics and Statistics,
P.O.\ Box 35 (MaD), FI-40014 University of Jyvaskyla}
\email{enrico.e.pasqualetto@jyu.fi}

\begin{document}
\date{\today} 
\begin{abstract}
We prove that the category of Banach \(L^0\)-modules over a given \(\sigma\)-finite measure space
is both complete and cocomplete, which means that it admits all small limits and colimits.
\end{abstract}
\maketitle
\tableofcontents
\section{Introduction}
In the well-established (but, nonetheless, still fast-growing) research field of analysis on metric measure spaces,
a significant role is played by the theory of Banach \(L^0\)-modules, which (in this context) was introduced by Gigli
in his seminal work \cite{Gigli14}. Therein, Banach \(L^0\)-modules are used to supply an abstract notion of
a `space of measurable tensor fields'. In this regard, an enlightening example is the so-called \emph{cotangent module},
which we are going to describe. An effective Sobolev theory on metric measure spaces is available
\cite{Cheeger00,HKST15,AmbrosioGigliSavare11}, so one can consider the `formal differentials' of Sobolev functions
also in this nonsmooth framework. However, it is clear that in order to obtain an arbitrary \(1\)-form this is not sufficient,
not even on differentiable manifolds: one should also have the freedom to multiply differentials by functions, to sum
the outcomes, and to take their limits. This corresponds to the fact that on a differentiable manifold \(M\) the smooth
\(1\)-forms can be obtained as limits of \(C^\infty(M)\)-linear combinations of differentials of smooth functions.
The line of thought described above led to the following axiomatisation in \cite{Gigli17}: a Banach \(L^0\)-module is
an algebraic module over the ring of \(L^0\)-functions (i.e.\ of measurable functions, quotiented up to a.e.\ equality)
endowed with a \emph{pointwise norm} that induces a complete distance; see Definitions \ref{def:normed_module} and
\ref{def:Banach_module} for the details. As it is evident from the literature on the topic, in order to achieve
a deeper understanding of the structure of metric measure spaces, it is of pivotal importance to put the
functional-analytic aspects of the tensor calculus via Banach \(L^0\)-modules on a firm ground. As an example
of this fact, we recall that the finite-dimensionality (or, more generally, the reflexivity) of the cotangent
module entails the density of Lipschitz functions in the Sobolev space. It is also worth pointing out that the
interest towards Banach \(L^0\)-modules goes far beyond the analysis on metric measure spaces. Indeed, the
essentially equivalent concept of a \emph{randomly normed space} was previously introduced in \cite{HLR91}, as a tool
for studying ultrapowers of Lebesgue--Bochner spaces over a rather general class of measure spaces. Later on,
the slightly different notion of a \emph{random normed module} was investigated (see \cite{Guo-2011} and the references therein):
the motivation comes from the theory of probabilistic metric spaces, and it has applications in finance optimisation
problems, with connections to the study of conditional and dynamic risk measures. Due to the above reasons, in this
paper we will consider Banach \(L^0\)-modules over an arbitrary \(\sigma\)-finite measure space. 
\medskip

The aim of the present paper is to study the category \({\bf BanMod}_{\mathbb X}\) of Banach \(L^0(\mathbb X)\)-modules,
where \(\mathbb X=(\X,\Sigma,\mm)\) is a given \(\sigma\)-finite measure space. The morphisms in \({\bf BanMod}_{\mathbb X}\)
are those \(L^0(\mathbb X)\)-linear operators \(\varphi\colon\mathscr M\to\mathscr N\) that satisfy
\(|\varphi(v)|\leq|v|\) for every \(v\in\mathscr M\). Our main result (namely, Theorem \ref{thm:bicompl})
states that \({\bf BanMod}_{\mathbb X}\) is both a complete category (i.e.\ all limits exist) and a cocomplete
category (i.e.\ all colimits exist). This means, in particular, that \({\bf BanMod}_{\mathbb X}\) admits all equalisers,
products, inverse limits, and pullbacks, as well as all coequalisers, coproducts, direct limits, and pushouts.
The existence of inverse and direct limits in \({\bf BanMod}_{\mathbb X}\) was already known:
\begin{itemize}
\item Inverse limits in \({\bf BanMod}_{\mathbb X}\), whose existence was proved in \cite{GPS18}, were
necessary to build the differential of a locally Sobolev map from a metric measure space to a metric space.
\item Direct limits in \({\bf BanMod}_{\mathbb X}\), whose existence was proved in the unpublished note
\cite{P19}, were used in \cite{DMLP19} to obtain a `representation theorem' for separable Banach \(L^0\)-modules.
Since each separable Banach \(L^0\)-module is the direct limit of a sequence of
finitely-generated Banach \(L^0\)-modules, a representation of an arbitrary separable Banach
\(L^0\)-module as the space of \(L^0\)-sections of a separable \emph{measurable Banach bundle}
could be deduced from the corresponding result for finitely-generated modules, which was previously
obtained in \cite{LP18}.
\end{itemize}
It is worth mentioning that the theory of Banach \(L^0\)-modules extends the one of Banach spaces,
as the latter correspond to Banach \(L^0\)-modules over a measure space whose measure is a Dirac delta.
In fact, the strategy of our proof of the (co)completeness of \({\bf BanMod}_{\mathbb X}\) is inspired
by the one of the category \({\bf Ban}\) of Banach spaces, for which we refer to \cite{SZ65,Pothoven68,Castillo10}.
However, some other aspects of the Banach \(L^0\)-module theory, as the \emph{inverse image functor}
(see Section \ref{ss:inv_img}), are characteristic of Banach \(L^0\)-modules and do not have a
(non-trivial) counterpart in the Banach space setting.
\medskip

We conclude by pointing out that the contents of this paper slightly overlap with those of the
unpublished note \cite{P19}. More specifically, the Examples \ref{ex:BanMod_not_balanced}, \ref{ex:inverse_trivial},
\ref{ex:inverse_not_right_exact}, \ref{ex:direct_not_left_exact}, and \ref{ex:InvIm_no_inverse}
are essentially taken from \cite{P19}, but besides them the two papers are in fact independent.
Indeed, the (co)completeness of \({\bf BanMod}_{\mathbb X}\) is proved \emph{without} using
the existence of inverse and direct limits.
\section{Preliminaries}
Throughout the paper, we denote by \(\mathcal P_F(I)\) the family of all finite subsets of a given set \(I\neq\varnothing\).
To avoid pathological situations, all the measure spaces we consider are assumed to be non-null.
We denote by \(\mathbb P=(P,\Sigma_P,\delta_p)\) the probability space made of a unique point \(p\),
where \(\Sigma_P=\{\varnothing,P\}\) is the only \(\sigma\)-algebra on \(P\) and
\(\delta_p\) is the Dirac measure at \(p\), i.e.\ \(\delta_p(\varnothing)=0\) and \(\delta_p(P)=1\).
Given two \(\sigma\)-finite measure spaces \(\mathbb X=(\X,\Sigma_\X,\mm_\X)\) and \(\mathbb Y=(\Y,\Sigma_\Y,\mm_\Y)\),
we define their product \(\mathbb X\times\mathbb Y\) as the \(\sigma\)-finite measure space
\(\big(\X\times\Y,\Sigma_\X\otimes\Sigma_\Y,\mm_\X\otimes\mm_\Y\big)\), where \(\Sigma_\X\otimes\Sigma_\Y\)
and \(\mm_\X\otimes\mm_\Y\) stand for the product \(\sigma\)-algebra and the product measure, respectively.
We tacitly identify \(\mathbb X\times\mathbb P\) with \(\mathbb X\).
\subsection{Normed and Banach \texorpdfstring{\(L^0(\mathbb X)\)}{L0(X)}-modules}
In this section, we present the theory of Banach \(L^0(\mathbb X)\)-modules, which was first introduced in
\cite{Gigli14} and then refined further in \cite{Gigli17}. See also \cite{GP20}.
\subsubsection*{The space \texorpdfstring{\(L^0(\mathbb X)\)}{L0(X)}}
Let \(\mathbb X=(\X,\Sigma,\mm)\) be a \(\sigma\)-finite measure space. We denote by \(L^0_{\rm ext}(\mathbb X)\)
the space of measurable functions from \((\X,\Sigma)\) to the extended real line \([-\infty,+\infty]\), quotiented up to \(\mm\)-a.e.\ equality.
The space \(L^0_{\rm ext}(\mathbb X)\) is a lattice if endowed with the following partial order relation: given any
\(f,g\in L^0_{\rm ext}(\mathbb X)\), we declare that \(f\leq g\) if and only if \(f(x)\leq g(x)\) holds for \(\mm\)-a.e.\ \(x\in\X\).
\begin{remark}\label{rmk:Dedekind_compl}{\rm
Recall that a lattice \((A,\leq)\) is said to be \emph{Dedekind complete} (see e.g.\ \cite{Fremlin74,Fremlin3})
if every non-empty subset of \(A\) having an upper bound
admits a least upper bound, or equivalently every non-empty subset of \(A\) having a lower bound admits a greatest
lower bound. It is well-known that the lattice \((L^0_{\rm ext}(\mathbb X),\leq)\) is order-bounded, is Dedekind complete, and satisfies
the following property: given any (possibly uncountable) non-empty subset \(\{f_i\}_{i\in I}\) of \(L^0_{\rm ext}(\mathbb X)\), there
exist two (at most) countable subsets \(C,C'\) of \(I\) such that \(\bigvee_{i\in I}f_i=\bigvee_{i\in C}f_i\) and
\(\bigwedge_{i\in I}f_i=\bigwedge_{i\in C'}f_i\).
\fr}\end{remark}

We define the Riesz space \(L^0(\mathbb X)\) as
\[
L^0(\mathbb X)\coloneqq\big\{f\in L^0_{\rm ext}(\mathbb X)\;\big|\;-\infty<f(x)<+\infty\,\text{ for }\mm\text{-a.e.\ }x\in\X\big\}.
\]
Then \(L^0(\mathbb X)\) is a commutative algebra with respect to the usual pointwise operations, as well as a
\(\sigma\)-sublattice of \(L^0_{\rm ext}(\mathbb X)\), i.e.\ a sublattice of \(L^0_{\rm ext}(\mathbb X)\) that
is closed under countable suprema and infima. In particular, Remark \ref{rmk:Dedekind_compl} ensures that
\((L^0(\mathbb X),\leq)\) is Dedekind complete and that each supremum (resp.\ infimum) in \(L^0(\mathbb X)\)
can be expressed as a countable supremum (resp.\ infimum). Moreover, the space \(L^0(\mathbb X)\) is a topological
algebra if endowed with the following complete distance:
\[
\sfd_{L^0(\mathbb X)}(f,g)\coloneqq\int|f-g|\wedge 1\,\d\tilde\mm\quad\text{ for every }f,g\in L^0(\mathbb X),
\]
where \(\tilde\mm\) is any finite measure on \((\X,\Sigma)\) satisfying \(\mm\ll\tilde\mm\ll\mm\). While the distance
\(\sfd_{L^0(\mathbb X)}\) depends on the specific choice of \(\tilde\mm\), its induced topology does not.
Moreover, a sequence \((f_n)_{n\in\N}\subseteq L^0(\mathbb X)\) satisfies \(\sfd_{L^0(\mathbb X)}(f_n,f)\to 0\)
as \(n\to\infty\) for some limit function \(f\in L^0(\mathbb X)\) if and only if we can extract a subsequence
\((n_i)_{i\in\N}\) such that \(f_{n_i}(x)\to f(x)\) as \(i\to\infty\) for \(\mm\)-a.e.\ \(x\in\X\).
\subsubsection*{Definition and main properties of Banach \texorpdfstring{\(L^0(\mathbb X)\)}{L0(X)}-modules}
We begin with the relevant definitions:
\begin{definition}[Normed \(L^0(\mathbb X)\)-module]\label{def:normed_module}
Let \(\mathbb X\) be a \(\sigma\)-finite measure space. Let \(\mathscr M\) be a module over \(L^0(\mathbb X)\).
Then we say that \(\mathscr M\) is a \emph{seminormed \(L^0(\mathbb X)\)-module} if it is endowed with a
\emph{pointwise seminorm}, i.e.\ with a mapping \(|\cdot|\colon\mathscr M\to L^0(\mathbb X)\) that satisfies the following properties:
\[\begin{split}
|v|\geq 0&\quad\text{ for every }v\in\mathscr M,\\
|v+w|\leq|v|+|w|&\quad\text{ for every }v,w\in\mathscr M,\\
|f\cdot v|=|f||v|&\quad\text{ for every }f\in L^0(\mathbb X)\text{ and }v\in\mathscr M.
\end{split}\]
Moreover, we say that \(\mathscr M\) is a \emph{normed \(L^0(\mathbb X)\)-module} provided \(|v|=0\) if and only if \(v=0\).
\end{definition}

Any pointwise seminorm on \(\mathscr M\) induces a pseudometric \(\sfd_{\mathscr M}\) on \(\mathscr M\):
\[
\sfd_{\mathscr M}(v,w)\coloneqq\sfd_{L^0(\mathbb X)}(|v-w|,0)\quad\text{ for every }v,w\in\mathscr M.
\]
It holds that \(\mathscr M\) is a normed \(L^0(\mathbb X)\)-module if and only if \(\sfd_{\mathscr M}\) is a distance.
\begin{definition}[Banach \(L^0(\mathbb X)\)-module]\label{def:Banach_module}
Let \(\mathbb X\) be a \(\sigma\)-finite measure space. Let \(\mathscr M\) be a normed \(L^0(\mathbb X)\)-module.
Then we say that \(\mathscr M\) is a \emph{Banach \(L^0(\mathbb X)\)-module} if the pointwise norm \(|\cdot|\)
is complete, meaning that the induced distance \(\sfd_{\mathscr M}\) is complete.
\end{definition}
The space \(L^0(\mathbb X)\) itself is a Banach \(L^0(\mathbb X)\)-module. Moreover, if the measure underlying \(\mathbb X\) is
a Dirac measure, then \(L^0(\mathbb X)\cong\R\), and so the Banach \(L^0(\mathbb X)\)-modules are exactly the Banach spaces.

\begin{remark}{\rm
A warning about the terminology: in this paper we distinguish between normed \(L^0(\mathbb X)\)-modules and Banach \(L^0(\mathbb X)\)-modules,
while in the original papers \cite{Gigli14,Gigli17} only complete normed \(L^0(\mathbb X)\)-modules were considered (but they were called just
`normed \(L^0\)-modules').
\fr}\end{remark}

Given a Banach \(L^0(\mathbb X)\)-module \(\mathscr M\), a partition \((E_n)_{n\in\N}\subseteq\Sigma\) of \(\X\), and a sequence
\((v_n)_{n\in\N}\subseteq\mathscr M\), the series \(\sum_{n\in\N}\nchi_{E_n}\cdot v_n\) converges unconditionally
in \(\mathscr M\) (by the dominated convergence theorem).
\begin{remark}\label{rmk:about_conv_L0}{\rm
Let \(\mathbb X=(\X,\Sigma,\mm)\) be a \(\sigma\)-finite measure space and \(\mathscr M\) a Banach \(L^0(\mathbb X)\)-module.
Then it follows from the properties of \(\sfd_{L^0(\mathbb X)}\) that for any partition \((E_n)_{n\in\N}\subseteq\Sigma\) of \(\X\) it holds that
\[
\mathscr M^\N\ni(v_n)_{n\in\N}\mapsto\sum_{n\in\N}\nchi_{E_n}\cdot v_n\in\mathscr M\quad\text{ is a continuous map,}
\]
where the source space is endowed with the product topology.
\fr}\end{remark}
\subsubsection*{Examples of Banach \texorpdfstring{\(L^0(\mathbb X)\)}{L0(X)}-modules}
There are many ways to obtain a Banach \(L^0(\mathbb X)\)-module:
\begin{itemize}
\item \emph{Submodule.} An \(L^0(\mathbb X)\)-submodule \(\mathscr N\) of a normed \(L^0(\mathbb X)\)-module \(\mathscr M\)
inherits the structure of a normed \(L^0(\mathbb X)\)-module (if endowed with the restriction of the pointwise norm of \(\mathscr M\)).
Moreover, if \(\mathscr M\) is complete and \(\mathscr N\) is closed in \(\mathscr M\), then \(\mathscr N\) is a Banach \(L^0(\mathbb X)\)-module.
\item \emph{Null space.} Let \(\mathscr M\), \(\mathscr N\) be Banach \(L^0(\mathbb X)\)-modules and \(\varphi\colon\mathscr M\to\mathscr N\)
an \(L^0(\mathbb X)\)-linear and continuous map. Then the space \(\varphi^{-1}(\{0\})\coloneqq\big\{v\in\mathscr M\,:\,\varphi(v)=0\big\}\)
is a Banach \(L^0(\mathbb X)\)-submodule of \(\mathscr M\).
\item \emph{Range.} Let \(\mathscr M\), \(\mathscr N\) be Banach \(L^0(\mathbb X)\)-modules and \(\varphi\colon\mathscr M\to\mathscr N\)
an \(L^0(\mathbb X)\)-linear and continuous map. Then the space \(\varphi(\mathscr M)\coloneqq\big\{\varphi(v)\,:\,v\in\mathscr M\big\}\) is a normed
\(L^0(\mathbb X)\)-submodule of \(\mathscr N\). We point out that, in general, the space \(\varphi(\mathscr M)\) is not complete
(see Example \ref{ex:BanMod_not_balanced}).
\item \emph{Metric identification.} Let \(\mathscr M\) be a seminormed \(L^0(\mathbb X)\)-module. The equivalence relation \(\sim_{|\cdot|}\)
on \(\mathscr M\) is defined as follows: given any \(v,w\in\mathscr M\), we declare that \(v\sim_{|\cdot|} w\) if and only if \(|v-w|=0\).
Then the quotient \(\mathscr M/\sim_{|\cdot|}\) inherits a structure of normed \(L^0(\mathbb X)\)-module.
\item \emph{Completion.} Let \(\mathscr M\) be a normed \(L^0(\mathbb X)\)-module. Then there exists a unique couple \((\bar{\mathscr M},\iota)\),
where \(\bar{\mathscr M}\) is a Banach \(L^0(\mathbb X)\)-module, while \(\iota\colon\mathscr M\to\bar{\mathscr M}\) is an \(L^0(\mathbb X)\)-linear
map that preserves the pointwise norm and satisfies \({\rm cl}_{\bar{\mathscr M}}(\iota(\mathscr M))=\bar{\mathscr M}\). Uniqueness is up
to a unique isomorphism: given any \((\mathscr N,\tilde\iota)\) with the same properties, there exists a unique \(L^0(\mathbb X)\)-linear
bijection \(\Phi\colon\bar{\mathscr M}\to\mathscr N\) that preserves the pointwise norm (i.e.\ an \emph{isomorphism} of Banach
\(L^0(\mathbb X)\)-modules) and satisfies the identity \(\Phi\circ\iota=\tilde\iota\).
\item \emph{Quotient.} Let \(\mathscr M\) be a Banach \(L^0(\mathbb X)\)-module and \(\mathscr N\) a Banach \(L^0(\mathbb X)\)-submodule of
\(\mathscr M\). Then the quotient \(\mathscr M/\mathscr N\) is a Banach \(L^0(\mathbb X)\)-module if endowed with the pointwise norm
\[
|v+\mathscr N|\coloneqq\bigwedge_{w\in\mathscr N}|v+w|\quad\text{ for every }v\in\mathscr M.
\]
\item \emph{Space of homomorphisms.} Let \(\mathscr M\), \(\mathscr N\) be Banach \(L^0(\mathbb X)\)-modules. Then
\(\textsc{Hom}(\mathscr M,\mathscr N)\) is defined as the space of \(L^0(\mathbb X)\)-linear maps \(\varphi\colon\mathscr M\to\mathscr N\) for which
there exists \(g\in L^0(\mathbb X)^+\) such that \(|\varphi(v)|\leq g|v|\) holds for every \(v\in\mathscr M\). If endowed with the pointwise norm
\[
|\varphi|\coloneqq\bigwedge\big\{g\in L^0(\mathbb X)^+\;\big|\;|\varphi(v)|\leq g|v|\,\text{ for every }v\in\mathscr M\big\}
\]
and the usual pointwise operations, the space \(\textsc{Hom}(\mathscr M,\mathscr N)\) is a Banach \(L^0(\mathbb X)\)-module.
Its elements are called the \emph{homomorphisms} of Banach \(L^0(\mathbb X)\)-modules from \(\mathscr M\) to \(\mathscr N\).
\item \emph{Dual.} The dual of a Banach \(L^0(\mathbb X)\)-module \(\mathscr M\) is defined as
\(\mathscr M^*\coloneqq\textsc{Hom}(\mathscr M,L^0(\mathbb X))\).
\item \emph{Hilbert modules.} Let \(S\neq\varnothing\) be an arbitrary set. We define the space \(\mathscr H_{\mathbb X}(S)\) as
\[
\mathscr H_{\mathbb X}(S)\coloneqq\bigg\{v\in L^0(\mathbb X)^S\;\bigg||v|\coloneqq\bigg(\sum_{s\in S}|v(s)|^2\bigg)^{1/2}\in L^0(\mathbb X)\bigg\}.
\]
In particular, given \(v\in\mathscr H_{\mathbb X}(S)\) we have \(v(s)=0\) for all but countably many \(s\in S\). Then
\(\big(\mathscr H_{\mathbb X}(S),|\cdot|\big)\) is a Banach \(L^0(\mathbb X)\)-module with respect to the componentwise operations.
Also, \(\mathscr H_{\mathbb X}(S)\) is a \emph{Hilbert \(L^0(\mathbb X)\)-module}, i.e.\ it verifies the pointwise parallelogram rule:
\[
|v+w|^2+|v-w|^2=2|v|^2+2|w|^2\quad\text{ for every }v,w\in\mathscr H_{\mathbb X}(S).
\]
The elements \(\{e_s\}_{s\in S}\subseteq\mathscr H_{\mathbb X}(S)\), defined as \(e_s(t)\coloneqq 0\) for every \(t\in S\setminus\{s\}\) and
\(e_s(s)\coloneqq\nchi_\X\), form an orthonormal basis of \(\mathscr H_{\mathbb X}(S)\). Hence, given any \(S_1,S_2\neq\varnothing\), it holds that
\(\mathscr H_{\mathbb X}(S_1)\) and \(\mathscr H_{\mathbb X}(S_2)\) are isomorphic as Banach \(L^0(\mathbb X)\)-modules if and only if
\({\rm card}(S_1)={\rm card}(S_2)\).
\item \emph{\(L^0\)-Lebesgue--Bochner space.}
Let \(\mathscr M\) be a Banach \(L^0(\mathbb Y)\)-module, for some \(\sigma\)-finite measure space \(\mathbb Y\).
Then we denote by \(L^0(\mathbb X;\mathscr M)\) the space of all measurable maps \(v\colon\X\to\mathscr M\)
taking values in a separable subspace of \(\mathscr M\) (that depends on \(v\)), quotiented up to \(\mm_\X\)-a.e.\ equality.
Then it holds that \(L^0(\mathbb X;\mathscr M)\) is a Banach \(L^0(\mathbb X\times\mathbb Y)\)-module if equipped with:
\[\begin{split}
(v+w)(x)\coloneqq v(x)+w(x)\in\mathscr M&\quad\text{ for every }v,w\in L^0(\mathbb X;\mathscr M)\text{ and }\mm_\X\text{-a.e.\ }x\in\X,\\
(f\cdot v)(x)\coloneqq f(x,\cdot)\cdot v(x)\in\mathscr M&\quad\text{ for every }f\in L^0(\mathbb X\times\mathbb Y),\,
v\in L^0(\mathbb X;\mathscr M),\text{ and }\mm_\X\text{-a.e.\ }x\in\X,\\
|v|(x,y)\coloneqq|v(x)|(y)&\quad\text{ for every }v\in L^0(\mathbb X;\mathscr M)
\text{ and }(\mm_\X\otimes\mm_\Y)\text{-a.e.\ }(x,y)\in\X\times\Y.
\end{split}\]
In particular, for any Banach space \(B\) we can regard \(L^0(\mathbb X;B)\) as a Banach \(L^0(\mathbb X)\)-module.
The space of \emph{simple maps} from \(\mathbb X\) to \(\mathscr M\), i.e.\ of those elements of \(L^0(\mathbb X;\mathscr M)\)
that can be written as \(\sum_{i=1}^n\nchi_{E_i}v_i\) for some \((E_i)_{i=1}^n\subseteq\Sigma_\X\) and
\((v_i)_{i=1}^n\subseteq\mathscr M\), is dense in \(L^0(\mathbb X;\mathscr M)\).
\end{itemize}
The above claims can be proved by adapting the arguments in the proofs of \cite[Section 1.2]{Gigli14}.
\begin{remark}\label{rmk:about_RNP}{\rm
Let \(\mathbb X\) be a \(\sigma\)-finite measure space and \(B\) a Banach space. We denote by \(B'\) the dual of \(B\)
in the sense of Banach spaces. Then the map \(\iota_{\mathbb X,B}\colon L^0(\mathbb X;B')\to L^0(\mathbb X;B)^*\), given by
\[
\iota_{\mathbb X,B}(\omega)(v)\coloneqq\omega(\cdot)\big(v(\cdot)\big)\in L^0(\mathbb X)
\quad\text{ for every }\omega\in L^0(\mathbb X;B')\text{ and }v\in L^0(\mathbb X;B),
\]
is a morphism of Banach \(L^0(\mathbb X)\)-modules that preserves the pointwise norm. It also holds that
\[
\iota_{\mathbb X,B}\;\text{ is an isomorphism}\quad\Longleftrightarrow\quad B'\;\text{ has the Radon--Nikod\'ym property};
\]
see \cite[Theorems 1.3.10 and 1.3.26]{AinBS}, \cite[Proposition 1.2.13]{Gigli14}, and \cite[Appendix B]{LP18}.
We point out that, even in the case where \(B'\) does not have the Radon--Nikod\'{y}m property, the space
\(L^0(\mathbb X;B)^*\) can be characterised in several ways, see e.g.\ \cite{GLP22} for (generalisations of) this kind of results.
\fr}\end{remark}
\subsection{A reminder on category theory}
In this section, we recall some important notions and results in category theory, mostly concerning limits and colimits.
We refer to \cite{mitchell1967theory,MacLane98,kashiwara2005categories,Awodey10} for a thorough account of these topics.
Let us begin by fixing some useful terminology and notation.
\medskip

Given a category \({\bf C}\), we denote by \({\rm Ob}_{\bf C}\) and \({\rm Hom}_{\bf C}\) the classes of its objects
and morphisms, respectively. The domain and the codomain of a morphism \(\varphi\colon X\to Y\) are denoted by
\({\rm dom}(\varphi)\coloneqq X\) and \({\rm cod}(\varphi)\coloneqq Y\), respectively. Given two objects \(X\), \(Y\)
of \({\bf C}\), we denote by \({\rm Hom}_{\bf C}(X,Y)\) the class of those morphisms \(\varphi\) in \({\bf C}\)
such that \({\rm dom}(\varphi)=X\) and \({\rm cod}(\varphi)=Y\). We say that \({\bf C}\) is \emph{small} if the classes
\({\rm Ob}_{\bf C}\) and \({\rm Hom}_{\bf C}\) are sets, while we say that \({\bf C}\) is \emph{locally small} if the
class \({\rm Hom}_{\bf C}(X,Y)\) is a set for every pair of objects \(X\), \(Y\) of \({\bf C}\). The \emph{opposite}
(or \emph{dual}) \emph{category} \({\bf C}^{\rm op}\) is obtained from \({\bf C}\) by `reversing the morphisms'
(see, for example, \cite[page 12]{kashiwara2005categories} for the precise definition of \({\bf C}^{\rm op}\)).
A category \({\bf D}\) is said to be a \emph{subcategory} of \({\bf C}\) provided the following conditions hold:
\begin{itemize}
\item \({\rm Ob}_{\bf D}\) is a subcollection of \({\rm Ob}_{\bf C}\).
\item For any two objects \(X\), \(Y\) of \({\bf D}\), the class \({\rm Hom}_{\bf D}(X,Y)\) is a subcollection of
\({\rm Hom}_{\bf C}(X,Y)\).
\item The composition in \({\bf D}\) is induced by the composition in \({\bf C}\).
\item The identity morphisms in \({\bf D}\) are identity morphisms in \({\bf C}\).
\end{itemize}
We say that \({\bf D}\) is a \emph{full subcategory} of \({\bf C}\) if \({\rm Hom}_{\bf D}(X,Y)\)
coincides with \({\rm Hom}_{\bf C}(X,Y)\) for every pair of objects \(X\), \(Y\) of \({\bf D}\).
Given two locally small categories \({\bf C}\), \({\bf D}\), each functor \(F\colon{\bf C}\to{\bf D}\) induces a collection
of maps \(F_{X,Y}\colon{\rm Hom}_{\bf C}(X,Y)\to{\rm Hom}_{\bf D}(F(X),F(Y))\) for every pair of objects \(X\), \(Y\) of
\({\bf C}\). We say that \(F\) is \emph{full} (resp.\ \emph{faithful}) if \(F_{X,Y}\) is surjective (resp.\ injective)
for every pair of objects \(X\), \(Y\) of \({\bf C}\). We say that \(F\) is \emph{injective on objects}
if \(F(X)\) and \(F(Y)\) are different whenever \(X\) and \(Y\) are different objects of \({\bf C}\).
We say that \({\bf C}\) \emph{has zero morphisms} if there is a collection of morphisms \(0_{XY}\colon X\to Y\), indexed
by the pairs \(X\), \(Y\) of objects of \({\bf C}\), that satisfies the following property: given any three objects \(X\), \(Y\),
\(Z\) of \({\bf C}\) and any two morphisms \(\varphi\colon Y\to Z\) and \(\psi\colon X\to Y\) in \({\bf C}\), the diagram
\[\begin{tikzcd}
X \arrow[d,swap,"\psi"] \arrow[drr,"0_{XZ}"] \arrow[rr,"0_{XY}"] & & Y \arrow[d,"\varphi"] \\
Y \arrow[rr,swap,"0_{YZ}"] & & Z
\end{tikzcd}\]
commutes. If \({\bf C}\) has zero morphisms, then the zero morphisms \(0_{XY}\) are uniquely determined.
\medskip

An object \(I\) of a category \({\bf C}\) is said to be an \emph{initial object} if for any object \(X\) of \({\bf C}\)
there exists exactly one morphism \(i_X\colon I\to X\). Dually, an object \(T\) of \({\bf C}\) is said to be a \emph{terminal object}
if for any object \(X\) of \({\bf C}\) there exists exactly one morphism \(t_X\colon X\to T\). An object that is both initial and
terminal is called a \emph{zero object} of \({\bf C}\). Initial and terminal objects (thus, a fortiori, zero objects)
are uniquely determined up to a unique isomorphism: given any two initial objects \(I_1\) and \(I_2\) of \({\bf C}\),
there exists a unique isomorphism \(I_1\to I_2\); similarly for terminal objects. A \emph{pointed category} is a category
\({\bf C}\) having a zero object, which we denote by \(0_{\bf C}\). Each pointed category \({\bf C}\) has zero morphisms:
for any two objects \(X\), \(Y\) of \({\bf C}\), the zero morphism \(0_{XY}\) is given by \(i_Y\circ t_X\).
\medskip

A morphism \(\varphi\colon X\to Y\) in a category \({\bf C}\) is said to be a \emph{monomorphism} if it is left-cancellative,
meaning that \(\psi_1=\psi_2\) holds whenever \(Z\) is an object of \({\bf C}\) and \(\psi_1,\psi_2\colon Z\to X\) are two
morphisms in \({\bf C}\) satisfying \(\varphi\circ\psi_1=\varphi\circ\psi_2\). Dually, a morphism \(\varphi\colon X\to Y\)
is said to be an \emph{epimorphism} if it is right-cancellative, meaning that \(\psi_1=\psi_2\) holds whenever \(Z\) is an
object of \({\bf C}\) and \(\psi_1,\psi_2\colon Y\to Z\) are two morphisms in \({\bf C}\) satisfying \(\psi_1\circ\varphi=\psi_2\circ\varphi\).
We say that a category \({\bf C}\) is \emph{balanced} if every morphism in \({\bf C}\) that is both a monomorphism and an epimorphism
is an isomorphism.
\medskip

Some important examples of categories, which play a key role in this paper, are the following:
\begin{itemize}
\item The category \({\bf Set}\), whose objects are the sets and whose morphisms are the functions.
\item The category \({\bf Meas}_\sigma\) of \(\sigma\)-finite measure spaces. Given any two \(\sigma\)-finite measure spaces
\(\mathbb X=(\X,\Sigma_\X,\mm_\X)\) and \(\mathbb Y=(\Y,\Sigma_\Y,\mm_\Y)\), a morphism \(\tau\colon\mathbb X\to\mathbb Y\)
is a \((\Sigma_\X,\Sigma_\Y)\)-measurable map \(\tau\colon\X\to\Y\) that satisfies \(\tau_\#\mm_\X\ll\mm_\Y\). This notion
of morphism differs from those of other authors, who require e.g.\ \(\tau\) to be measure-preserving or to verify \(\tau_\#\mm_\X\leq C\mm_\Y\)
for some constant \(C>0\). Notice also that the measure \(\tau_\#\mm_\X\) on \(\Y\) needs not be \(\sigma\)-finite.
\item The category \({\bf Ban}\), whose objects are the Banach spaces and whose morphisms are the linear \(1\)-Lipschitz operators.
We refer e.g.\ to \cite{Pothoven68,Castillo10} for a study of the category \({\bf Ban}\).
\item Let \((I,\leq)\) be a \emph{directed set}, i.e.\ a non-empty partially ordered set where any pair of elements admits an upper bound.
Then \((I,\leq)\) can be regarded as a small category, where the objects are the elements of \(I\), while the morphisms are as follows:
given any \(i,j\in I\), we declare that there is a (unique) morphism \(i\to j\) if and only if \(i\leq j\).
\end{itemize}

Given two categories \({\bf C}\), \({\bf D}\) and two functors \(F,G\colon{\bf D}\to{\bf C}\), a \emph{natural transformation}
from \(F\) to \(G\) is a collection \(\eta_\star\) of morphisms \(\eta_X\colon F(X)\to G(X)\), indexed by the objects \(X\) of
\({\bf D}\), such that for any morphism \(\varphi\colon X\to Y\) in \({\bf D}\) the diagram
\[\begin{tikzcd}
F(X) \arrow[d,swap,"F(\varphi)"] \arrow[r,"\eta_X"] & G(X) \arrow[d,"G(\varphi)"] \\
F(Y) \arrow[r,swap,"\eta_Y"] & G(Y)
\end{tikzcd}\]
commutes. We denote by \({\bf C}^{\bf D}\) the \emph{functor category} from \({\bf D}\) to \({\bf C}\), whose objects are the functors
from \({\bf D}\) to \({\bf C}\) and whose morphisms are the natural transformations between them. An isomorphism in \({\bf C}^{\bf D}\) is called
a \emph{natural isomorphism}. Given an index category \({\bf J}\), the \emph{diagonal functor} \(\Delta_{{\bf J},{\bf C}}\colon{\bf C}\to{\bf C}^{\bf J}\)
is defined as follows: given any object \(X\) of \({\bf C}\), we set \(\Delta_{{\bf J},{\bf C}}(X)(i)\coloneqq X\) for every object
\(i\) of \({\bf J}\) and \(\Delta_{{\bf J},{\bf C}}(X)(\phi)\coloneqq{\rm id}_X\) for every morphism \(\phi\) in \({\bf J}\);
given any morphism \(\varphi\colon X\to Y\) in \({\bf C}\), we define the natural transformation \(\Delta_{{\bf J},{\bf C}}(\varphi)_\star\)
as \(\Delta_{{\bf J},{\bf C}}(\varphi)_i\coloneqq\varphi\) for every object \(i\) of \({\bf J}\).
\medskip

The \emph{arrow category} \({\bf C}^\to\) of a given category \({\bf C}\) is defined as follows. The objects of \({\bf C}^\to\)
are the morphisms in \({\bf C}\), while a morphism \(\varphi\to\psi\) in \({\bf C}^\to\) is given by a couple \((\alpha,\beta)\)
of morphisms \(\alpha\colon{\rm dom}(\varphi)\to{\rm dom}(\psi)\) and \(\beta\colon{\rm cod}(\varphi)\to{\rm cod}(\psi)\) in
\({\bf C}\) for which the following diagram commutes:
\[\begin{tikzcd}
{\rm dom}(\varphi) \arrow[d,swap,"\varphi"] \arrow[r,"\alpha"] & {\rm dom}(\psi) \arrow[d,"\psi"] \\
{\rm cod}(\varphi) \arrow[r,swap,"\beta"] & {\rm cod}(\psi)
\end{tikzcd}\]

More generally, given three categories \({\bf C}\), \({\bf C}_1\), \({\bf C}_2\) and two functors \(F\colon{\bf C}_1\to{\bf C}\)
and \(G\colon{\bf C}_2\to{\bf C}\), we define the \emph{comma category} \((F\downarrow G)\) in the following way:
\begin{itemize}
\item The objects of \((F\downarrow G)\) are the triples \((X,Y,\varphi)\), where \(X\) is an object of \({\bf C}_1\), \(Y\) is
an object of \({\bf C}_2\), and \(\varphi\colon F(X)\to G(Y)\) is a morphism in \({\bf C}\).
\item A morphism \((X_1,Y_1,\varphi_1)\to(X_2,Y_2,\varphi_2)\) in \((F\downarrow G)\) is given by a couple \((\psi_1,\psi_2)\),
where \(\psi_1\colon X_1\to X_2\) is a morphism in \({\bf C}_1\) and \(\psi_2\colon Y_1\to Y_2\) is a morphism in \({\bf C}_2\)
such that
\[\begin{tikzcd}
F(X_1) \arrow[d,swap,"\varphi_1"] \arrow[r,"F(\psi_1)"] & F(X_2) \arrow[d,"\varphi_2"] \\
G(Y_1) \arrow[r,swap,"G(\psi_2)"] & G(Y_2)
\end{tikzcd}\]
is a commutative diagram.
\end{itemize}
If \({\bf C}_2={\bf 1}\) (i.e.\ \({\bf C}_2\) is the one-object one-morphism category) and \(X\) is a given object of \({\bf C}\),
we just write \((F\downarrow X)\) instead of \((F\downarrow\Delta_{{\bf 1},{\bf C}}(X))\). Similarly for \((X\downarrow G)\) in the 
case where \({\bf C}_1={\bf 1}\). Observe also that the arrow category \({\bf C}^\to\) coincides with the comma category
\(({\rm id}_{\bf C}\downarrow{\rm id}_{\bf C})\).
\subsubsection*{Limits and colimits}
Let \({\bf J}\) be an index category. Then a \emph{diagram} of type \({\bf J}\) in a category \({\bf C}\) is a functor
\(D\colon{\bf J}\to{\bf C}\). A \emph{cone} to \(D\) is an object \(X\) of \({\bf C}\) together with a collection \(\varphi_\star\) of
morphisms \(\varphi_i\colon X\to D(i)\), indexed by the objects \(i\) of \({\bf J}\), such that \(D(\phi)\circ\varphi_i=\varphi_j\)
for every morphism \(\phi\colon i\to j\) in \({\bf J}\). A \emph{limit} of \(D\) is a cone \((L,\lambda_\star)\) to \(D\) that satisfies
the following universal property: given any cone \((X,\varphi_\star)\) to \(D\), there exists a unique morphism \(\Phi\colon X\to L\)
such that the diagram
\[\begin{tikzcd}
& X \arrow[dddl,swap,"\varphi_i"] \arrow[dd,dashed,"\Phi"] \arrow[dddr,"\varphi_j"] & \\
& & & \\
& L \arrow[dl,"\lambda_i"] \arrow[dr,swap,"\lambda_j"] & \\
D(i) \arrow[rr,swap,"D(\phi)"] & & D(j)
\end{tikzcd}\]
commutes for every morphism \(\phi\colon i\to j\) in \({\bf J}\). If a limit of \(D\) exists, then it is \emph{essentially} unique
(i.e.\ unique up to a unique isomorphism), so that we are entitled to refer to it as `the' limit of \(D\).
Alternatively, the cones of \(D\) can be identified with the objects of the category \((\Delta_{{\bf J},{\bf C}}\downarrow D)\), which is
thus called the \emph{category of cones} to \(D\), and a limit of \(D\) is a terminal object in \((\Delta_{{\bf J},{\bf C}}\downarrow D)\).
\medskip

Dually, by a \emph{cocone} of \(D\) we mean an object \(X\) of \({\bf C}\) together with a collection \(\varphi_\star\) of
morphisms \(\varphi_i\colon D(i)\to X\), indexed by the objects \(i\) of \({\bf J}\), such that \(\varphi_j\circ D(\phi)=\varphi_i\)
for every morphism \(\phi\colon i\to j\) in \({\bf J}\). A \emph{colimit} of \(D\) is a cocone \((C,c_\star)\) of \(D\) that satisfies
the following universal property: given any cocone \((X,\varphi_\star)\) of \(D\), there exists a unique morphism \(\Phi\colon C\to X\)
such that the diagram
\[\begin{tikzcd}
D(i) \arrow[dddr,swap,"\varphi_i"] \arrow[dr,"c_i"] \arrow[rr,"D(\phi)"] & & D(j) \arrow[dl,swap,"c_j"] \arrow[dddl,"\varphi_j"] \\
& C \arrow[dd,dashed,"\Phi"] & \\
& & & \\
& X &
\end{tikzcd}\]
commutes for every morphism \(\phi\colon i\to j\) in \({\bf J}\). Whenever it exists, a colimit of \(D\) is essentially unique,
thus we can unambiguously call it `the' colimit of \(D\). Alternatively, the cocones of \(D\) can be identified with the objects of the category
\((D\downarrow\Delta_{{\bf J},{\bf C}})\), which is thus called the \emph{category of cocones} of \(D\), and a colimit of \(D\) is an
initial object in \((D\downarrow\Delta_{{\bf J},{\bf C}})\).
\medskip

The following are some of the most important examples of limits in a category \({\bf C}\):
\begin{itemize}
\item Let \(X\), \(Y\) be two objects of \({\bf C}\) and \(a,b\colon X\to Y\) two morphisms in \({\bf C}\). Then the \emph{equaliser} of
\(a,b\colon X\to Y\) is an object \({\rm Eq}(a,b)\) of \({\bf C}\) together with a morphism \({\rm eq}(a,b)\colon{\rm Eq}(a,b)\to X\)
with \(a\circ{\rm eq}(a,b)=b\circ{\rm eq}(a,b)\) verifying the following universal property: if \(u\colon E\to X\) is a morphism in \({\bf C}\)
satisfying \(a\circ u=b\circ u\), then there exists a unique morphism \(\Phi\colon E\to{\rm Eq}(a,b)\) such that \({\rm eq}(a,b)\circ\Phi=u\).
The equaliser \(\big({\rm Eq}(a,b),{\rm eq}(a,b)\big)\) coincides with the limit of the diagram of type \({\bf J}_\rightrightarrows\) in
\({\bf C}\), where \({\bf J}_\rightrightarrows\) is the category made of two objects (corresponding to \(X\) and \(Y\)) and
(besides the identity morphisms) having only two parallel morphisms
between them (corresponding to \(a\) and \(b\)).
Observe also that \({\rm eq}(a,b)\) is a monomorphism.
\item Suppose \({\bf C}\) has zero morphisms. Then the \emph{kernel} of a morphism \(\varphi\colon X\to Y\) in \({\bf C}\) is
\[
\big({\rm Ker}(\varphi),{\rm ker}(\varphi)\big)\coloneqq\big({\rm Eq}(\varphi,0_{XY}),{\rm eq}(\varphi,0_{XY})\big),
\]
whenever the equaliser of \(\varphi,0_{XY}\colon X\to Y\) exists.
\item Let \(X_\star=\{X_i\}_{i\in I}\) be a set of objects of \({\bf C}\). Then the \emph{product} of \(X_\star\)
in \({\bf C}\) is an object \(\prod^{\bf C}X_\star=\prod_{i\in I}^{\bf C}X_i\) together with a family of morphisms
\(\big\{\pi_i\colon\prod^{\bf C}X_\star\to X_i\big\}_{i\in I}\) verifying the following universal property: given an object
\(Y\) of \({\bf C}\) and a family of morphisms \(\{\varphi_i\colon Y\to X_i\}_{i\in I}\), there exists a unique morphism
\(\Phi\colon Y\to\prod^{\bf C}X_\star\) such that \(\pi_i\circ\Phi=\varphi_i\) for every \(i\in I\). The product
\(\big(\prod^{\bf C}X_\star,\{\pi_i\}_{i\in I}\big)\) coincides with the limit of the diagram of type \({\bf J}_I\)
in \({\bf C}\), where \({\bf J}_I\) is the discrete category whose objects are the elements of \(I\).
\item Let \((I,\leq)\) be a directed set. By an \emph{inverse} (or \emph{projective}) \emph{system} in \({\bf C}\) indexed by \(I\)
we mean a family of objects \(\{X_i\}_{i\in I}\), together with a family \(\{{\rm P}_{ij}\,:\,i,j\in I,\,i\leq j\}\) of morphisms
\({\rm P}_{ij}\colon X_j\to X_i\) such that \({\rm P}_{ii}={\rm id}_{X_i}\) for every \(i\in I\) and \({\rm P}_{ik}={\rm P}_{ij}\circ{\rm P}_{jk}\)
for every \(i,j,k\in I\) with \(i\leq j\leq k\). Then the \emph{inverse} (or \emph{projective}) \emph{limit} of
\(\big(\{X_i\}_{i\in I},\{{\rm P}_{ij}\}_{i\leq j}\big)\) is an object \(\varprojlim X_\star\) of \({\bf C}\),
together with a family \(\{{\rm P}_i\}_{i\in I}\) of morphisms \({\rm P}_i\colon\varprojlim X_\star\to X_i\) such that
\({\rm P}_{ij}\circ{\rm P}_j={\rm P}_i\) for every \(i\leq j\) and verifying the following universal property:
given an object \(Y\) of \({\bf C}\) and morphisms \({\rm Q}_i\colon Y\to X_i\) with \({\rm P}_{ij}\circ{\rm Q}_j={\rm Q}_i\)
for every \(i\leq j\), there exists a unique morphism \(\Phi\colon Y\to\varprojlim X_\star\) such that \({\rm P}_i\circ\Phi={\rm Q}_i\)
for every \(i\in I\). Moreover, \(\big(\varprojlim X_\star,\{{\rm P}_i\}_{i\in I}\big)\) coincides with the limit
of the diagram of type \((I,\leq)^{\rm op}\) in \({\bf C}\).
\item Let \(X\), \(Y\), \(Z\) be objects of \({\bf C}\). Let \(\varphi_X\colon X\to Z\) and \(\varphi_Y\colon Y\to Z\) be two given morphisms.
Then the \emph{pullback} of \(\varphi_X\) and \(\varphi_Y\) is an object \(P=X\times_Z Y\) of \({\bf C}\), together with two morphisms
\(p_X\colon P\to X\) and \(p_Y\colon P\to Y\) with \(\varphi_X\circ p_X=\varphi_Y\circ p_Y\) verifying the
following universal property: given an object \(Q\) of \({\bf C}\) and morphisms \(q_X\colon Q\to X\) and \(q_Y\colon Q\to Y\)
with \(\varphi_X\circ q_X=\varphi_Y\circ q_Y\), there exists a unique morphism \(\Phi\colon Q\to P\) such that
\[\begin{tikzcd}
Q \arrow[ddddr,bend right=20,swap,"q_X"] \arrow[ddr,dashed,"\Phi"] \arrow[ddrrr,bend left=20,"q_Y"] & & &\\
& & & \\
& P \arrow[dd,swap,"p_X"] \arrow[rr,"p_Y"] & & Y \arrow[dd,"\varphi_Y"] \\
& & & \\
& X \arrow[rr,swap,"\varphi_X"] & & Z
\end{tikzcd}\]
is a commutative diagram. The pullback \((X\times_Z Y,p_X,p_Y)\) coincides with the limit of the diagram of type
\({\bf J}_\lrcorner\) in \({\bf C}\), where \({\bf J}_\lrcorner\) is the category with three objects
(corresponding to \(X\), \(Y\), and \(Z\)) and whose morphisms (besides the identity ones) are \(X\to Z\) and \(Y\to Z\).
\end{itemize}

Dually, the following are some of the most important examples of colimits in a category \({\bf C}\):
\begin{itemize}
\item Let \(X\), \(Y\) be two objects of \({\bf C}\) and \(a,b\colon X\to Y\) two morphisms in \({\bf C}\). Then the \emph{coequaliser} of
\(a,b\colon X\to Y\) is given by an object \({\rm Coeq}(a,b)\) of \({\bf C}\) together with a morphism \({\rm coeq}(a,b)\colon Y\to{\rm Coeq}(a,b)\)
with \({\rm coeq}(a,b)\circ a={\rm coeq}(a,b)\circ b\) verifying the following universal property: if \(u\colon Y\to F\) is a morphism in \({\bf C}\)
satisfying \(u\circ a=u\circ b\), then there exists a unique morphism \(\Phi\colon{\rm Coeq}(a,b)\to F\) such that \(\Phi\circ{\rm coeq}(a,b)=u\).
The coequaliser \(\big({\rm Coeq}(a,b),{\rm coeq}(a,b)\big)\) coincides with the colimit of the diagram of type \({\bf J}_\rightrightarrows\) in \({\bf C}\).
Observe also that \({\rm coeq}(a,b)\) is an epimorphism.
\item Suppose \({\bf C}\) has zero morphisms. Then the \emph{cokernel} of a morphism \(\varphi\colon X\to Y\) in \({\bf C}\) is
\[
\big({\rm Coker}(\varphi),{\rm coker}(\varphi)\big)\coloneqq\big({\rm Coeq}(\varphi,0_{XY}),{\rm coeq}(\varphi,0_{XY})\big),
\]
whenever the coequaliser of \(\varphi,0_{XY}\colon X\to Y\) exists.
\item Let \(X_\star=\{X_i\}_{i\in I}\) be a set of objects of \({\bf C}\). Then the \emph{coproduct} of \(X_\star\)
in \({\bf C}\) is an object \(\coprod^{\bf C}X_\star=\coprod_{i\in I}^{\bf C}X_i\) together with a family of morphisms
\(\big\{\iota_i\colon X_i\to\coprod^{\bf C}X_\star\big\}_{i\in I}\) verifying the following universal property: given an object
\(Y\) of \({\bf C}\) and a family of morphisms \(\{\varphi_i\colon X_i\to Y\}_{i\in I}\), there exists a unique morphism
\(\Phi\colon\coprod^{\bf C}X_\star\to Y\) such that \(\Phi\circ\iota_i=\varphi_i\) for every \(i\in I\). The coproduct
\(\big(\coprod^{\bf C}X_\star,\{\iota_i\}_{i\in I}\big)\) coincides with the colimit of the diagram of type \({\bf J}_I\)
in \({\bf C}\).
\item Let \((I,\leq)\) be a directed set. By a \emph{direct} (or \emph{inductive}) \emph{system} in \({\bf C}\) indexed by \(I\)
we mean a family of objects \(\{X_i\}_{i\in I}\), together with a family \(\{\varphi_{ij}\,:\,i,j\in I,\,i\leq j\}\) of morphisms
\(\varphi_{ij}\colon X_i\to X_j\) such that \(\varphi_{ii}={\rm id}_{X_i}\) for every \(i\in I\) and \(\varphi_{ik}=\varphi_{jk}\circ\varphi_{ij}\)
for every \(i,j,k\in I\) with \(i\leq j\leq k\). Then the \emph{direct} (or \emph{inductive}) \emph{limit} of
\(\big(\{X_i\}_{i\in I},\{\varphi_{ij}\}_{i\leq j}\big)\) is an object \(\varinjlim X_\star\) of \({\bf C}\),
together with a family \(\{\varphi_i\}_{i\in I}\) of morphisms \(\varphi_i\colon X_i\to\varinjlim X_\star\) such that
\(\varphi_j\circ\varphi_{ij}=\varphi_i\) for every \(i\leq j\) and verifying the following universal property:
given an object \(Y\) of \({\bf C}\) and morphisms \(\psi_i\colon X_i\to Y\) with \(\psi_j\circ\varphi_{ij}=\psi_i\)
for every \(i\leq j\), there exists a unique morphism \(\Phi\colon\varinjlim X_\star\to Y\) such that \(\Phi\circ\varphi_i=\psi_i\)
for every \(i\in I\). Moreover, \(\big(\varinjlim X_\star,\{\varphi_i\}_{i\in I}\big)\) coincides with the colimit
of the diagram of type \((I,\leq)\) in \({\bf C}\).
\item Let \(X\), \(Y\), \(Z\) be objects of \({\bf C}\). Let \(\varphi_X\colon Z\to X\) and \(\varphi_Y\colon Z\to Y\) be two given morphisms.
Then the \emph{pushout} of \(\varphi_X\) and \(\varphi_Y\) is an object \(P=X\sqcup_Z Y\) of \({\bf C}\), together with two morphisms
\(i_X\colon X\to P\) and \(i_Y\colon Y\to P\) with \(i_X\circ\varphi_X=i_Y\circ\varphi_Y\) verifying the
following universal property: given an object \(Q\) of \({\bf C}\) and morphisms \(j_X\colon X\to Q\) and \(j_Y\colon Y\to Q\)
with \(j_X\circ\varphi_X=j_Y\circ\varphi_Y\), there exists a unique morphism \(\Phi\colon P\to Q\) such that the diagram
\[\begin{tikzcd}
Z \arrow[dd,swap,"\varphi_X"] \arrow[rr,"\varphi_Y"] & & Y \arrow[dd,"i_Y"] \arrow[ddddr,bend left=20, "j_Y"] & \\
& & & \\
X \arrow[rr,swap,"i_X"] \arrow[ddrrr,bend right=20,swap,"j_X"] & & P \arrow[ddr,dashed,,swap,"\Phi"] & \\
& & & \\
& & & Q
\end{tikzcd}\]
commutes. The pushout \((X\sqcup_Z Y,i_X,i_Y)\) coincides with the colimit of the diagram of type \({\bf J}_\ulcorner\)
in \({\bf C}\), where by \({\bf J}_\ulcorner\) we mean the category with three objects (corresponding to \(X\), \(Y\),
and \(Z\)) whose morphisms (besides the identity ones) are \(Z\to X\) and \(Z\to Y\).
\end{itemize}
\begin{remark}{\rm
A warning about the terminology: differently from other authors, we only consider (co)products that are indexed by sets (not by classes).
These are sometimes called \emph{small (co)products}. Moreover, by a projective (resp.\ an inductive) limit
we mean the limit (resp.\ the colimit) of a diagram that is indexed by a directed set, while for some authors `projective limit'
(resp.\ `inductive limit') is a synonim of (not necessarily directed) `limit' (resp.\ `colimit').
\fr}\end{remark}

Next, we recall the notion of image and coimage of a morphism. Fix a category \({\bf C}\) that has zero morphisms and
admits all finite limits and colimits. Let \(\varphi\colon X\to Y\) be a morphism in \({\bf C}\). Then:
\begin{itemize}
\item The \emph{image} of \(\varphi\) is given by \(\big({\rm Im}(\varphi),{\rm im}(\varphi)\big)\coloneqq\big({\rm Eq}(i_1,i_2),{\rm eq}(i_1,i_2)\big)\),
where \((Y\sqcup_X Y,i_1,i_2)\) is the pushout of \(\varphi\) and \(\varphi\).
\item The \emph{coimage} of \(\varphi\) is given by \(\big({\rm Coim}(\varphi),{\rm coim}(\varphi)\big)\coloneqq\big({\rm Coeq}(p_1,p_2),{\rm coeq}(p_1,p_2)\big)\),
where \((Y\times_X Y,p_1,p_2)\) is the pullback of \(\varphi\) and \(\varphi\).
\end{itemize}
Given a category \({\bf C}\) having zero morphisms, we say that a monomorphism \(\varphi\colon X\to Y\) in \({\bf C}\) is \emph{normal}
if there exist an object \(Z\) and a morphism \(\eta\colon Y\to Z\) such that \((X,\varphi)\cong\big({\rm Ker}(\eta),{\rm ker}(\eta)\big)\).
Dually, we say that an epimorphism \(\psi\colon X\to Y\) in \({\bf C}\) is \emph{conormal} if there exist an object \(W\) and a morphism
\(\theta\colon W\to X\) such that \((Y,\psi)\cong\big({\rm Coker}(\theta),{\rm coker}(\theta)\big)\). The category \({\bf C}\) is said
to be \emph{normal} if every monomorphism in \({\bf C}\) is normal, \emph{conormal} if every epimorphism in \({\bf C}\) is conormal.
A category that is either normal or conormal is balanced, see e.g.\ \cite[Proposition 14.3]{mitchell1967theory}.
\subsubsection*{Completeness and cocompleteness}
A category \({\bf C}\) is said to be \emph{complete} if all small limits in \({\bf C}\) (i.e.\ limits of diagrams whose index category is small) exist.
Dually, \({\bf C}\) is said to be \emph{cocomplete} if all small colimits in \({\bf C}\) exist. A category that is both complete and cocomplete is
called a \emph{bicomplete category}. An example of a bicomplete category is \({\bf Set}\). The product \(\prod_{i\in I}^{\bf Set}X_i\) of a given family
of sets \(X_\star=\{X_i\}_{i\in I}\) is the Cartesian product \(\prod_{i\in I}X_i\), while the coproduct \(\coprod_{i\in I}^{\bf Set}X_i\) is the
disjoint union \(\bigsqcup_{i\in I}X_i\). Another example of a bicomplete category is \({\bf Ban}\), see for example \cite{SZ65}.
\medskip

The following result provides a criterion to detect complete and cocomplete categories. Albeit well-known to the experts, we report its proof for the usefulness of the reader.
\begin{theorem}\label{thm:suff_(co)complete}
A category in which all products and equalisers exist is complete. A category in which all coproducts and coequalisers exist is cocomplete.
\end{theorem}
\begin{proof}
We prove only the first part of the statement, as the second one follows by a dual argument. Let \({\bf C}\) be a category
in which all products and equalisers exist. Let \(D\colon{\bf J}\to{\bf C}\) be a small diagram. Then the products
\(\big(\prod_{i\in{\rm Ob}_{\bf J}}^{\bf C}D(i),\pi_\star\big)\) and
\(\big(\prod_{\phi\in{\rm Hom}_{\bf J}}^{\bf C}D({\rm cod}(\phi)),\eta_\star\big)\) exist in \({\bf C}\).
For brevity, let us set \(Y\coloneqq\prod_{i\in{\rm Ob}_{\bf J}}^{\bf C}D(i)\) and
\(Z\coloneqq\prod_{\phi\in{\rm Hom}_{\bf J}}^{\bf C}D({\rm cod}(\phi))\).
Observe that there exist two (uniquely determined) morphisms \(a,b\colon Y\to Z\) such that
\[\begin{tikzcd}
Y \arrow[r,dashed,"a"] \arrow[dr,swap,"\pi_{{\rm cod}(\phi)}"] & Z \arrow[d,"\eta_\phi"] &
Y \arrow[l,swap,dashed,"b"] \arrow[dl,"D(\phi)\circ\pi_{{\rm dom}(\phi)}"] \\
& D({\rm cod}(\phi)) &
\end{tikzcd}\]
is a commutative diagram for every \(\phi\in{\rm Hom}_{\bf J}\).
Let us define \(\lambda_i\coloneqq\pi_i\circ{\rm eq}(a,b)\colon{\rm Eq}(a,b)\to D(i)\) for every \(i\in{\rm Ob}_{\bf J}\).
Given a morphism \(\phi\colon i\to j\) in \({\bf J}\), one has that
\[\begin{split}
D(\phi)\circ\lambda_i&=D(\phi)\circ\pi_{{\rm dom}(\phi)}\circ{\rm eq}(a,b)=\eta_\phi\circ b\circ{\rm eq}(a,b)=\eta_\phi\circ a\circ{\rm eq}(a,b)\\
&=\pi_{{\rm cod}(\phi)}\circ{\rm eq}(a,b)=\lambda_j.
\end{split}\]
This shows that \(({\rm Eq}(a,b),\lambda_\star)\) is a cone to \(D\). Moreover, if \((X,\varphi_\star)\) is a cone to \(D\), then there exists
a unique morphism \(u\colon X\to Y\) such that \(\pi_i\circ u=\varphi_i\) for every \(i\in{\rm Ob}_{\bf J}\). In particular, the diagram
\[\begin{tikzcd}
X \arrow[r,"u"] \arrow[ddr,swap,"\varphi_{{\rm cod}(\phi)}"] & Y \arrow[r,"a"] \arrow[dd,"\pi_{{\rm cod}(\phi)}"]
& Z \arrow[dd,"\eta_\phi"] & Y \arrow[l,swap,"b"] \arrow[dd,swap,"\pi_{{\rm dom}(\phi)}"] & X \arrow[l,swap,"u"] \arrow[ddl,"\varphi_{{\rm dom}(\phi)}"] \\
& & & & \\
& D({\rm cod}(\phi)) \arrow[r,equal] & D({\rm cod}(\phi)) & D({\rm dom}(\phi)) \arrow[l,"D(\phi)"] &
\end{tikzcd}\]
commutes for every \(\phi\in{\rm Hom}_{\bf J}\). Since \(D(\phi)\circ\varphi_{{\rm dom}(\phi)}=\varphi_{{\rm cod}(\phi)}\), we deduce that
\(a\circ u=b\circ u\), thus there exists a unique morphism \(\Phi\colon X\to{\rm Eq}(a,b)\) such that \({\rm eq}(a,b)\circ\Phi=u\). Observe that
\[
\lambda_i\circ\Phi=\pi_i\circ{\rm eq}(a,b)\circ\Phi=\pi_i\circ u=\varphi_i\quad\text{ for every }i\in{\rm Ob}_{\bf J}.
\]
Finally, we claim that \(\Phi\) is the unique morphism satisfying \(\lambda_i\circ\Phi=\varphi_i\) for every \(i\in{\rm Ob}_{\bf J}\).
To prove it, fix any morphism \(\Psi\colon X\to{\rm Eq}(a,b)\) such that \(\lambda_i\circ\Psi=\varphi_i\) for every \(i\in{\rm Ob}_{\bf J}\).
This means that \(\pi_i\circ{\rm eq}(a,b)\circ\Psi=\varphi_i\) holds for every \(i\in{\rm Ob}_{\bf J}\), which forces \({\rm eq}(a,b)\circ\Psi=u\)
by the uniqueness of \(u\), and thus \(\Psi=\Phi\) by the uniqueness of \(\Phi\). All in all, \(({\rm Eq}(a,b),\lambda_\star)\) is the limit of \(D\).
\end{proof}
\subsubsection*{Limits and colimits as functors}
If \({\bf J}\) is a small index category and all the diagrams of type \({\bf J}\) in a category \({\bf C}\) have
limits, then there exists a unique functor \({\rm Lim}_{{\bf J},{\bf C}}\colon{\bf C}^{\bf J}\to{\bf C}\), called
the \emph{limit functor} from \({\bf J}\) to \({\bf C}\), such that \({\rm Lim}_{{\bf J},{\bf C}}(D)\) is (the object
underlying) the limit of \(D\) for any diagram \(D\colon{\bf J}\to{\bf C}\). Dually, if all the diagrams of type \({\bf J}\)
in \({\bf C}\) have colimits, then there exists a unique functor \({\rm Colim}_{{\bf J},{\bf C}}\colon{\bf C}^{\bf J}\to{\bf C}\),
called the \emph{colimit functor} from \({\bf J}\) to \({\bf C}\), such that \({\rm Colim}_{{\bf J},{\bf C}}(D)\) is (the object
underlying) the colimit of \(D\) for any diagram \(D\colon{\bf J}\to{\bf C}\). In the special cases of inverse and direct limits,
we write \(\varprojlim_I\) and \(\varinjlim_I\) instead of \({\rm Lim}_{(I,\leq)^{\rm op},{\bf C}}\) and
\({\rm Colim}_{(I,\leq),{\bf C}}\), respectively.
\begin{remark}\label{rmk:kernel_of_natural_hom}{\rm
Let \({\bf C}\) be a category having zero morphisms and where all kernels exist, and \({\bf J}\) a small index category.
Let \(\theta_\star\colon D_1\to D_2\) be a morphism in \({\bf C}^{\bf J}\). Then the kernel of \(\theta_\star\)
in \({\bf C}^{\bf J}\) exists:
\begin{itemize}
\item The diagram \({\rm Ker}(\theta_\star)\colon{\bf J}\to{\bf C}\) is given by \({\rm Ker}(\theta_\star)(i)={\rm Ker}(\theta_i)\)
for every \(i\in{\rm Ob}_{\bf J}\), and for any morphism  \(\phi\colon i\to j\) in \({\bf J}\) we have that
\({\rm Ker}(\theta_\star)(\phi)\colon{\rm Ker}(\theta_i)\to{\rm Ker}(\theta_j)\) is the unique morphism \(\Phi\)
satisfying \({\rm ker}(\theta_j)\circ\Phi=D_1(\phi)\circ{\rm ker}(\theta_i)\). Both the existence and the uniqueness
of \(\Phi\) are consequences of the commutativity of the following diagram:
\[\begin{tikzcd}
{\rm Ker}(\theta_i) \arrow[d,swap,dashed,"\Phi"] \arrow[r,"{\rm ker}(\theta_i)"] & D_1(i) \arrow[d,swap,"D_1(\phi)"] \arrow[r,"\theta_i"]
& D_2(i) \arrow[d,"D_2(\phi)"] \arrow[r] & \{0\} \arrow[d,equal] \\
{\rm Ker}(\theta_j) \arrow[r,swap,"{\rm ker}(\theta_j)"] & D_1(j) \arrow[r,swap,"\theta_j"] & D_2(j) \arrow{r} & \{0\}
\end{tikzcd}\]
\item The morphism \({{\rm ker}(\theta_\star)}\) coincides with the natural transformation
\({{\rm ker}(\theta_\star)}_\star\colon{\rm Ker}(\theta_\star)\to D_1\), which is defined as
\({{\rm ker}(\theta_\star)}_i\coloneqq{\rm ker}(\theta_i)\colon{\rm Ker}(\theta_i)\to D_1(i)\) for every \(i\in{\rm Ob}_{\bf J}\).
\end{itemize}
The dual statement holds for the cokernel of \(\theta_\star\) (in the case where all cokernels exist in \({\bf C}\)).
\fr}\end{remark}
Given two categories \({\bf C}\), \({\bf D}\), and two functors \(F\colon{\bf C}\to{\bf D}\) and \(G\colon{\bf D}\to{\bf C}\),
we say that \(F\) is the \emph{right adjoint} to \(G\), or that \(G\) is the \emph{left adjoint} to \(F\) (for short \(G\dashv F\)),
if for any object \(X\) of \({\bf C}\) there exists a morphism \(\varepsilon_X\colon G(F(X))\to X\) with the following property: given
an object \(Y\) of \({\bf D}\) and a morphism \(\varphi\colon G(Y)\to X\), there exists a unique morphism \(\psi\colon Y\to F(X)\) such that
\[\begin{tikzcd}
G(Y) \arrow[d,dashed,swap,"F(\psi)"] \arrow[dr,"\varphi"] & \\
G(F(X)) \arrow[r,swap,"\varepsilon_X"] & X
\end{tikzcd}\]
commutes.
Each right adjoint functor \(F\colon{\bf C}\to{\bf D}\) is \emph{continuous}, i.e.\ it preserves limits: if the limit \((L,\lambda_\star)\) of
a diagram \(D\colon{\bf J}\to{\bf C}\) exists, then the limit \((L^F,\lambda^F_\star)\) of the diagram \(F\circ D\colon{\bf J}\to{\bf D}\) exists
as well, and the unique morphism \(\Phi\colon F(L)\to L^F\) such that
\[\begin{tikzcd}
F(L) \arrow[dr,swap,"F(\lambda_i)"] \arrow[r,dashed,"\Phi"] & L^F \arrow[d,"\lambda^F_i"] \\
& F(D(i))
\end{tikzcd}\]
is a commutative diagram for every \(i\in{\rm Ob}_{\bf J}\) is an isomorphism. In particular, each right adjoint functor is \emph{left exact},
i.e.\ it commutes with finite limits. Dually, each left adjoint functor is \emph{cocontinuous}, i.e.\ it preserves colimits: if the colimit
\((C,c_\star)\) of a diagram \(D\colon{\bf J}\to{\bf C}\) exists, then also the colimit \((C^F,c^F_\star)\) of \(F\circ D\colon{\bf J}\to{\bf D}\)
exists, and the unique morphism \(\Psi\colon C^F\to F(C)\) such that
\[\begin{tikzcd}
F(D(i)) \arrow[d,swap,"c^F_i"] \arrow[dr,"F(c_i)"] & \\
C^F \arrow[r,dashed,swap,"\Psi"] & F(C)
\end{tikzcd}\]
is a commutative diagram for every \(i\in{\rm Ob}_{\bf J}\) is an isomorphism.
In particular, each left adjoint functor is \emph{right exact}, i.e.\ it commutes with finite colimits.
If \({\bf C}\) is a complete (resp.\ cocomplete) category, then for any small index category \({\bf J}\) it holds that
\(\Delta_{{\bf J},{\bf C}}\dashv{\rm Lim}_{{\bf J},{\bf C}}\) (resp.\ \({\rm Colim}_{{\bf J},{\bf C}}\dashv\Delta_{{\bf J},{\bf C}}\)),
thus in particular \({\rm Lim}_{{\bf J},{\bf C}}\) is continuous (resp.\ \({\rm Colim}_{{\bf J},{\bf C}}\) is cocontinuous).
\section{The category of Banach \texorpdfstring{\(L^0(\mathbb X)\)}{L0(X)}-modules}
Let \(\mathbb X\) be a given \(\sigma\)-finite measure space. Then we denote by \({\bf BanMod}_{\mathbb X}\) the category of Banach \(L^0(\mathbb X)\)-modules,
where a morphism \(\varphi\colon\mathscr M\to\mathscr N\) between two Banach \(L^0(\mathbb X)\)-modules \(\mathscr M\) and \(\mathscr N\) is defined as an
\(L^0(\mathbb X)\)-linear operator (i.e.\ a homomorphism of \(L^0(\mathbb X)\)-modules) that satisfies
\[
|\varphi(v)|\leq|v|\quad\text{ for every }v\in\mathscr M.
\]
The morphisms in \({\rm Hom}_{{\bf BanMod}_{\mathbb X}}(\mathscr M,\mathscr N)\) are exactly those
\(\varphi\in\textsc{Hom}(\mathscr M,\mathscr N)\) satisfying \(|\varphi|\leq 1\). Also, the isomorphisms
in \({\bf BanMod}_{\mathbb X}\) are exactly the isomorphisms of Banach \(L^0(\mathbb X)\)-modules.
\begin{remark}{\rm
One could also consider the category \({\bf BanMod}_{\mathbb X}^0\), where the morphisms between two Banach \(L^0(\mathbb X)\)-modules
\(\mathscr M\) and \(\mathscr N\) are given exactly by the elements of \(\textsc{Hom}(\mathscr M,\mathscr N)\). Observe that in this category
a morphism \(\varphi\colon\mathscr M\to\mathscr N\) is an isomorphism if and only if it is bijective and there exist
\(g_1,g_2\in L^0(\mathbb X)\) with \(g_2\geq g_1>0\) such that \(g_1|v|\leq|\varphi(v)|\leq g_2|v|\)
for every \(v\in\mathscr M\). Our choice of working with \({\bf BanMod}_{\mathbb X}\) is due to the fact that,
trivially, arbitrary limits and colimits cannot exist in \({\bf BanMod}_{\mathbb X}^0\).
However, by suitably adapting the results we are going to present, one can show that \({\bf BanMod}_{\mathbb X}^0\) is \emph{finitely bicomplete}
(i.e.\ it admits finite limits and colimits).
\fr}\end{remark}
\subsection{Basic properties of \texorpdfstring{\({\bf BanMod}_{\mathbb X}\)}{BanModX}}
Let \(\mathbb X\) be a \(\sigma\)-finite measure space. Then it can be readily checked that \({\bf BanMod}_{\mathbb X}\)
is a pointed category, whose zero object \(0_{{\bf BanMod}_{\mathbb X}}\) is given by the trivial Banach \(L^0(\mathbb X)\)-module
consisting uniquely of the zero element. Moreover, \({\bf BanMod}_{\mathbb X}\) is locally small: given two Banach \(L^0(\mathbb X)\)-modules
\(\mathscr M\) and \(\mathscr N\), we observed that the morphisms \(\mathscr M\to\mathscr N\) form a subset of
\(\textsc{Hom}(\mathscr M,\mathscr N)\), thus in particular \({\rm card}\big({\rm Hom}_{{\bf BanMod}_{\mathbb X}}(\mathscr M,\mathscr N)\big)
\leq{\rm card}(\mathscr N)^{{\rm card}(\mathscr M)}\).
\begin{proposition}
Let \(\mathbb X\) be a \(\sigma\)-finite measure space. Then the category \({\bf BanMod}_{\mathbb X}\) is not small.
\end{proposition}
\begin{proof}
Given any cardinal \(\kappa\), we fix a set \(S_\kappa\) of cardinality \(\kappa\). Recall that if \(\kappa_1\), \(\kappa_2\) are different
cardinals, then \(\mathscr H_{\mathbb X}(S_{\kappa_1})\) is not isomorphic to \(\mathscr H_{\mathbb X}(S_{\kappa_2})\). Since the collection
of all cardinals is a proper class (i.e.\ it is not a set), we conclude that the collection of all Hilbert \(L^0(\mathbb X)\)-modules is a proper
class as well. This implies that the category \({\bf BanMod}_{\mathbb X}\) is not small, as we claimed.
\end{proof}
\begin{proposition}
Let \(\mathbb X\) be a \(\sigma\)-finite measure space and \(\varphi\colon\mathscr M\to\mathscr N\)
a morphism in \({\bf BanMod}_{\mathbb X}\). Then the following properties hold:
\begin{itemize}
\item \(\varphi\) is a monomorphism if and only if it is injective.
\item \(\varphi\) is an epimorphism if and only if \(\varphi(\mathscr M)\) is dense in \(\mathscr N\).
\end{itemize}
\end{proposition}
\begin{proof}
Suppose \(\varphi\) is a monomorphism. Denote by \(\psi_1\colon\varphi^{-1}(\{0\})\to\mathscr M\) and
\(\psi_2\colon\varphi^{-1}(\{0\})\to\mathscr M\) the inclusion map and the null map, respectively.
Since both \(\varphi\circ\psi_1\) and \(\varphi\circ\psi_2\) coincide with the null map from
\(\varphi^{-1}(\{0\})\) to \(\mathscr N\), we deduce that \(\psi_1=\psi_2\) and thus \(\varphi^{-1}(\{0\})=\{0\}\),
whence the injectivity of \(\varphi\) follows by linearity. Conversely, suppose \(\varphi\) is injective. Given any
Banach \(L^0(\mathbb X)\)-module \(\mathscr P\) and any two morphisms \(\psi_1,\psi_2\colon\mathscr P\to\mathscr M\)
satisfying \(\varphi\circ\psi_1=\varphi\circ\psi_2\), we have that \(\psi_1=\psi_2\), otherwise there would exist
an element \(z\in\mathscr P\) such that \(\psi_1(z)\neq\psi_2(z)\) and thus accordingly
\(\varphi\big(\psi_1(z)\big)\neq\varphi\big(\psi_2(z)\big)\) by the injectivity of \(\varphi\).
This shows that \(\varphi\) is a monomorphism.

Suppose \(\varphi\) is an epimorphism. Define \(\mathscr Q\coloneqq\mathscr N/{\rm cl}_{\mathscr N}(\varphi(\mathscr M))\).
Denote by \(\psi_1\colon\mathscr N\to\mathscr Q\) the canonical projection map on the quotient and by
\(\psi_2\colon\mathscr N\to\mathscr Q\) the null map. Observe that
\[
\psi_1\big(\varphi(v)\big)=\varphi(v)+{\rm cl}_{\mathscr N}(\varphi(\mathscr M))
={\rm cl}_{\mathscr N}(\varphi(\mathscr M))\quad\text{ for every }v\in\mathscr M,
\]
thus \(\psi_1\circ\varphi\) and \(\psi_2\circ\varphi\) are the null map. Hence, we have
\(\psi_1=\psi_2\), which implies \({\rm cl}_{\mathscr N}(\varphi(\mathscr M))=\mathscr N\).
Conversely, suppose \(\varphi(\mathscr M)\) is dense in \(\mathscr N\). Choose any Banach
\(L^0(\mathbb X)\)-module \(\mathscr P\) and any two morphisms \(\psi_1,\psi_2\colon\mathscr N\to\mathscr P\)
with \(\psi_1\circ\varphi=\psi_2\circ\varphi\). This means that \(\psi_1=\psi_2\) on the dense subspace
\(\varphi(\mathscr M)\) of \(\mathscr N\), thus accordingly \(\psi_1=\psi_2\) on the whole \(\mathscr N\).
This shows that \(\varphi\) is an epimorphism.
\end{proof}
\begin{example}\label{ex:BanMod_not_balanced}{\rm
Let \(\mathbb X\) be a given \(\sigma\)-finite measure space. Recall that \(c_0\) stands for the space of all sequences
\((a_n)_{n\in\N}\in\R^\N\) such that \(a_n\to 0\) as \(n\to\infty\), and that \(c_0\) is a separable Banach space if
endowed with the componentwise operations and the supremum norm \(\big\|(a_n)_{n\in\N}\big\|_{c_0}\coloneqq\sup_{n\in\N}|a_n|\).
Thanks to the density of simple maps in the Banach \(L^0(\mathbb X)\)-module \(L^0(\mathbb X;c_0)\), there exists a unique
morphism \(\varphi\colon L^0(\mathbb X;c_0)\to L^0(\mathbb X;c_0)\) in the category \({\bf BanMod}_{\mathbb X}\) such that
\[
\varphi\big(\nchi_\X(a_n)_{n\in\N}\big)=\nchi_\X(a_n/n)_{n\in\N}\quad\text{ for every }(a_n)_{n\in\N}\in c_0.
\]
Clearly, \(\varphi\) is injective and thus a monomorphism in \({\bf BanMod}_{\mathbb X}\). Moreover, the range of \(\varphi\) contains all
simple maps from \(\mathbb X\) to \(c_{00}\), where \(c_{00}\) stands for the space consisting of those sequences \((a_n)_{n\in\N}\)
such that \(a_n=0\) holds for all but finitely many \(n\in\N\). Since \(c_{00}\) is dense in \(c_0\), it follows that \(\varphi(L^0(\mathbb X;c_0))\)
is dense in \(L^0(\mathbb X;c_0)\) and thus \(\varphi\) is an epimorphism in \({\bf BanMod}_{\mathbb X}\). However, we have that the element
\(\nchi_\X(1/n)_{n\in\N}\in L^0(\mathbb X;c_0)\) does not belong to \(\varphi(L^0(\mathbb X;c_0))\), which means that \(\varphi\) is not surjective.
In particular, the morphism \(\varphi\) is not an isomorphism in \({\bf BanMod}_{\mathbb X}\).
\fr}\end{example}
\begin{remark}{\rm
Example \ref{ex:BanMod_not_balanced} shows that, given a \(\sigma\)-finite measure space \(\mathbb X\), the category
\({\bf BanMod}_{\mathbb X}\) is not balanced, which implies that it is neither normal nor conormal.
\fr}\end{remark}
\subsection{Limits and colimits in \texorpdfstring{\({\bf BanMod}_{\mathbb X}\)}{BanModX}}
In this section, we prove that \({\bf BanMod}_{\mathbb X}\) is bicomplete.
\subsubsection*{Kernels and cokernels in \({\bf BanMod}_{\mathbb X}\)}
We begin by proving that (co)kernels exist in \({\bf BanMod}_{\mathbb X}\):
\begin{theorem}[Kernels in \({\bf BanMod}_{\mathbb X}\)]\label{thm:kernel_BanMod}
Let \(\mathbb X\) be a \(\sigma\)-finite measure space and \(\varphi\colon\mathscr M\to\mathscr N\) a morphism between Banach
\(L^0(\mathbb X)\)-modules \(\mathscr M\), \(\mathscr N\). Then the kernel of \(\varphi\) exists and is given by
\[
{\rm Ker}(\varphi)\cong\varphi^{-1}(\{0\}),
\]
together with the inclusion map \({\rm ker}(\varphi)\colon\varphi^{-1}(\{0\})\to\mathscr M\).
In particular, the equaliser of any two morphisms \(\varphi,\psi\colon\mathscr M\to\mathscr N\) exists and is given by
\begin{equation}\label{eq:equaliser_BanMod}
\big({\rm Eq}(\varphi,\psi),{\rm eq}(\varphi,\psi)\big)\cong
\bigg({\rm Ker}\Big(\frac{\varphi-\psi}{2}\Big),{\rm ker}\Big(\frac{\varphi-\psi}{2}\Big)\bigg).
\end{equation}
\end{theorem}
\begin{proof}
We know that \(\varphi^{-1}(\{0\})\) is a Banach \(L^0(\mathbb X)\)-submodule of \(\mathscr M\) and
that the inclusion map \(\iota\colon\varphi^{-1}(\{0\})\to\mathscr M\) is a morphism of Banach
\(L^0(\mathbb X)\)-modules with \(\varphi\circ\iota=0\). Now fix a Banach \(L^0(\mathbb X)\)-module
\(\mathscr E\) and a morphism \(u\colon\mathscr E\to\mathscr M\) such that \(\varphi\circ u=0\).
Define \(\Phi\colon\mathscr E\to\varphi^{-1}(\{0\})\) as \(\Phi(w)\coloneqq u(w)\) for every \(w\in\mathscr E\).
Note that \(\Phi\) is the unique map from \(\mathscr E\) to \(\varphi^{-1}(\{0\})\) with
\(\iota\circ\Phi=u\). Since \(\Phi\) is a morphism, we conclude that the kernel of \(\varphi\) is \(\varphi^{-1}(\{0\})\)
together with the inclusion map. This proves the first part of the statement, whence \eqref{eq:equaliser_BanMod} immediately follows.
\end{proof}
\begin{theorem}[Cokernels in \({\bf BanMod}_{\mathbb X}\)]\label{thm:cokernel_BanMod}
Let \(\mathbb X\) be a \(\sigma\)-finite measure space and \(\varphi\colon\mathscr M\to\mathscr N\) a morphism between Banach
\(L^0(\mathbb X)\)-modules \(\mathscr M\), \(\mathscr N\). Then the cokernel of \(\varphi\) exists and is given by
\[
{\rm Coker}(\varphi)\cong\mathscr N/{\rm cl}_{\mathscr N}\big(\varphi(\mathscr M)\big),
\]
together with the canonical projection map on the quotient \({\rm coker}(\varphi)\colon\mathscr N\to\mathscr N/{\rm cl}_{\mathscr N}\big(\varphi(\mathscr M)\big)\).
In particular, the coequaliser of any two morphisms \(\varphi,\psi\colon\mathscr M\to\mathscr N\) exists and is given by
\begin{equation}\label{eq:coequaliser_BanMod}
\big({\rm Coeq}(\varphi,\psi),{\rm coeq}(\varphi,\psi)\big)\cong
\bigg({\rm Coker}\Big(\frac{\varphi-\psi}{2}\Big),{\rm coker}\Big(\frac{\varphi-\psi}{2}\Big)\bigg).
\end{equation}
\end{theorem}
\begin{proof}
We know that the quotient \(\mathscr Q\coloneqq\mathscr N/{\rm cl}_{\mathscr N}\big(\varphi(\mathscr M)\big)\) is a Banach \(L^0(\mathbb X)\)-module
and that the canonical projection \(\pi\colon\mathscr N\to\mathscr Q\) is a morphism of Banach \(L^0(\mathbb X)\)-modules. Notice that
\(\pi\circ\varphi=0\). Now fix a Banach \(L^0(\mathbb X)\)-module \(\mathscr F\) and a morphism \(u\colon\mathscr N\to\mathscr F\)
such that \(u\circ\varphi=0\). Define the mapping \(\Phi\colon\mathscr Q\to\mathscr F\) as \(\Phi([w])\coloneqq u(w)\)
for every \(w\in\mathscr N\), where \([w]\in\mathscr Q\) stands for the equivalence class of \(w\). Notice also that
\({\rm cl}_{\mathscr N}\big(\varphi(\mathscr M)\big)\subseteq u^{-1}(\{0\})\). This implies that if two elements \(v,\tilde v\in\mathscr M\)
satisfy \(v-\tilde v\in{\rm cl}_{\mathscr N}\big(\varphi(\mathscr M)\big)\), then \(u(v)-u(\tilde v)=u(v-\tilde v)=0\).
Hence, \(\Phi\) is well-defined. Moreover, if \(w\in\mathscr N\) and \(z\in u^{-1}(\{0\})\),
then \(|u(w)|=|u(w+z)|\leq|w+z|\), which gives
\[
\big|\Phi([w])\big|=|u(w)|\leq\bigwedge_{z\in u^{-1}(\{0\})}|w+z|\leq
\bigwedge_{z\in{\rm cl}_{\mathscr N}(\varphi(\mathscr M))}|w+z|=|[w]|\quad\text{ for every }[w]\in\mathscr Q.
\]
Since \(\Phi\) is \(L^0(\mathbb X)\)-linear by construction, we deduce that it is a morphism of Banach \(L^0(\mathbb X)\)-modules.
Observe that \(\Phi\colon\mathscr Q\to\mathscr F\) is the unique map satisfying \(\Phi\circ\pi=u\). Hence, the cokernel of
\(\varphi\) is given by \(\mathscr Q\) together with the canonical projection map on the quotient. This proves the first part of the statement,
whence \eqref{eq:coequaliser_BanMod} immediately follows.
\end{proof}
\begin{remark}{\rm
Observe that in Theorems \ref{thm:kernel_BanMod} and \ref{thm:cokernel_BanMod} we cannot write \({\rm Eq}(\varphi,\psi)={\rm Ker}(\varphi-\psi)\)
or \({\rm Coeq}(\varphi,\psi)={\rm Coker}(\varphi-\psi)\), since \(\varphi-\psi\) needs not be a morphism in \({\bf BanMod}_{\mathbb X}\).
This is due to the fact that, in general, \(\big|(\varphi-\psi)(v)\big|\leq 2|v|\) for every \(v\in\mathscr M\) is the best inequality one
can have. In particular, we have that \({\bf BanMod}_{\mathbb X}\) is not enriched over the category of Abelian groups.
\fr}\end{remark}
\subsubsection*{Products and coproducts in \({\bf BanMod}_{\mathbb X}\)}
To construct (co)products in the category of Banach \(L^0(\mathbb X)\)-modules,
we need to introduce the notion of \(\ell_p\)-sum of a family of Banach \(L^0(\mathbb X)\)-modules.
\begin{definition}[\(\ell_p\)-sum]
Let \(\mathbb X\) be a \(\sigma\)-finite measure space and \(\mathscr M_\star=\{\mathscr M_i\}_{i\in I}\) a family of Banach
\(L^0(\mathbb X)\)-modules. Let \(p\in[1,\infty]\) be fixed. Given any \(v_\star=(v_i)_{i\in I}\in\prod_{i\in I}^{\bf Set}\mathscr M_i\), we define
\[
|v_\star|_p\coloneqq\left\{\begin{array}{ll}
\bigvee\big\{\big(\sum_{i\in J}|v_i|^p\big)^{1/p}\;\big|\;J\in\mathcal P_F(I)\big\}\in L^0_{\rm ext}(\mathbb X)\\
\bigvee\big\{|v_i|\;\big|\;i\in I\big\}\in L^0_{\rm ext}(\mathbb X)
\end{array}\quad\begin{array}{ll}
\text{ if }p<\infty,\\
\text{ if }p=\infty.
\end{array}\right.
\]
Then we define the \emph{\(\ell_p\)-sum} of \(\mathscr M_\star\) as
\[
\ell_p(\mathscr M_\star)\coloneqq\bigg\{v_\star\in\prod_{i\in I}^{\bf Set}\mathscr M_i\;\bigg|\;|v_\star|_p\in L^0(\mathbb X)\bigg\}.
\]
In the case where \(I\) consists of finitely many elements, say that \(\mathscr M_\star=\{\mathscr M_1,\ldots,\mathscr M_n\}\), we denote
\[
\mathscr M_1\oplus_p\dots\oplus_p\mathscr M_n\coloneqq\ell_p(\mathscr M_\star).
\]
\end{definition}
Rather standard verifications show that \(\ell_p(\mathscr M_\star)\) has a Banach \(L^0(\mathbb X)\)-module structure:
\begin{proposition}\label{prop:suff_cond_complete}
Let \(\mathbb X\) be a \(\sigma\)-finite measure space and \(\mathscr M_\star=\{\mathscr M_i\}_{i\in I}\) a family of Banach
\(L^0(\mathbb X)\)-modules. Let \(p\in[1,\infty]\) be fixed. Then \(\big(\ell_p(\mathscr M_\star),|\cdot|_p\big)\) is a
Banach \(L^0(\mathbb X)\)-module with respect to the following operations:
\[\begin{split}
v_\star+w_\star\coloneqq(v_i+w_i)_{i\in I}&\quad\text{ for every }v_\star,w_\star\in\ell_p(\mathscr M_\star),\\
f\cdot v_\star\coloneqq(f\cdot v_i)_{i\in I}&\quad\text{ for every }f\in L^0(\mathbb X)\text{ and }v_\star\in\ell_p(\mathscr M_\star).
\end{split}\]
\end{proposition}
\begin{proof}
First of all, we aim to prove that for any \(v_\star,w_\star\in\ell_p(\mathscr M_\star)\) and \(f\in L^0(\mathbb X)\) it holds that
\begin{equation}\label{eq:suff_cond_complete}
|v_\star+w_\star|_p\leq|v_\star|_p+|w_\star|_p,\qquad|f\cdot v_\star|_p=|f||v_\star|_p.
\end{equation}
We discuss only the case \(p<\infty\), as the case \(p=\infty\) is easier. Fix any \(J\in\mathcal P_F(I)\) and notice that
\[\begin{split}
\bigg(\sum_{i\in J}|v_i+w_i|^p\bigg)^{1/p}&\leq\bigg(\sum_{i\in J}|v_i|^p\bigg)^{1/p}+\bigg(\sum_{i\in J}|w_i|^p\bigg)^{1/p}\leq|v_\star|_p+|w_\star|_p,\\
\bigg(\sum_{i\in J}|f\cdot v_i|^p\bigg)^{1/p}&=|f|\bigg(\sum_{i\in J}|v_i|^p\bigg)^{1/p}\leq|f||v_\star|_p.
\end{split}\]
By the arbitrariness of \(J\in\mathcal P_F(I)\), we deduce that \(|v_\star+w_\star|_p\leq|v_\star|_p+|w_\star|_p\)
and \(|f\cdot v_\star|_p\leq|f||v_\star|_p\), whence it follows that
\(|f|\big||v_\star|_p-|\nchi_{\{f\neq 0\}}\cdot v_\star|_p\big|\leq|f||\nchi_{\{f=0\}}\cdot v_\star|_p\leq\nchi_{\{f=0\}}|f||v_\star|_p=0\).
Therefore, letting \(g\coloneqq\nchi_{\{f\neq 0\}}\frac{1}{f}\in L^0(\mathbb X)\), we can estimate
\[
|f||v_\star|_p=|f||\nchi_{\{f\neq 0\}}\cdot v_\star|_p=|f||(fg)\cdot v_\star|_p\leq|f||g||f\cdot v_\star|_p\leq|f\cdot v_\star|_p.
\]
All in all, \eqref{eq:suff_cond_complete} is proved. In particular, \(v_\star+w_\star\) and \(f\cdot v_\star\) belong to
\(\ell_p(\mathscr M_\star)\) for every \(v_\star,w_\star\in\ell_p(\mathscr M_\star)\) and \(f\in L^0(\mathbb X)\). It can
be readily checked that \(\big(\ell_p(\mathscr M_\star),|\cdot|_p\big)\) is a normed \(L^0(\mathbb X)\)-module. It remains to verify that
\(\sfd_{\ell_p(\mathscr M_\star)}\) is a complete distance. Let \((v^n_\star)_{n\in\N}\subseteq\ell_p(\mathscr M_\star)\) be a Cauchy sequence.
Given any \(i\in I\), we have \(|v_i^n-v_i^m|\leq|v_\star^n-v_\star^m|_p\) for every \(n,m\in\N\), whence it follows that
\((v_i^n)_{n\in\N}\subseteq\mathscr M_i\) is a Cauchy sequence. Let \(v_i\in\mathscr M_i\) denote its limit as \(n\to\infty\).
Moreover, \(\big||v_\star^n|_p-|v_\star^m|_p\big|\leq|v_\star^n-v_\star^m|_p\) for every \(n,m\in\N\), thus \((|v_\star^n|_p)_{n\in\N}\)
is a Cauchy sequence in \(L^0(\mathbb X)\), which converges to some limit function \(g\in L^0(\mathbb X)\). For each \(J\in\mathcal P_F(I)\)
we have that \(\big(\sum_{i\in J}|v_i^n|^p\big)^{1/p}\leq|v_\star^n|_p\) for every \(n\in\N\), so that by letting \(n\to\infty\) we deduce
that \(\big(\sum_{i\in J}|v_i|^p\big)^{1/p}\leq g\). Given that \(J\in\mathcal P_F(I)\) was arbitrary, we obtain that \(|v_\star|_p\leq g\),
thus in particular \(v_\star\in\ell_p(\mathscr M_\star)\). Finally, for any \(\varepsilon>0\) we can find \(\bar n\in\N\) such that
\(\sfd_{L^0(\mathbb X)}(|v_\star^n-v_\star^m|,0)\leq\varepsilon\) for every \(n,m\geq\bar n\). Since \((|v_\star^n-v_\star^m|)_{m\geq\bar n}\)
is a Cauchy sequence in \(L^0(\mathbb X)\), it converges to some function \(g_n\in L^0(\mathbb X)\) as \(m\to\infty\). Notice that
\(\sfd_{L^0(\mathbb X)}(g_n,0)\leq\varepsilon\). For each \(J\in\mathcal P_F(I)\) and \(n,m\geq\bar n\), we have
\(\big(\sum_{i\in J}|v_i^n-v_i^m|^p\big)^{1/p}\leq|v_\star^n-v_\star^m|_p\), thus by letting \(m\to\infty\) we deduce that
\(\big(\sum_{i\in J}|v_i^n-v_i|^p\big)^{1/p}\leq g_n\). This implies that \(|v_\star^n-v_\star|_p\leq g_n\), and accordingly
that \(\sfd_{\ell_p(\mathscr M_\star)}(v_\star^n,v_\star)\leq\varepsilon\), for every \(n\geq\bar n\). Therefore, \(v_\star^n\to v_\star\)
in \(\ell_p(\mathscr M_\star)\) as \(n\to\infty\), as desired.
\end{proof}

With the concept of \(\ell_p\)-sum at disposal, we can describe products and coproducts in \({\bf BanMod}_{\mathbb X}\):
\begin{theorem}[Products in \({\bf BanMod}_{\mathbb X}\)]\label{thm:prod_BanMod}
Let \(\mathbb X\) be a \(\sigma\)-finite measure space, \(\mathscr M_\star=\{\mathscr M_i\}_{i\in I}\) a set of
Banach \(L^0(\mathbb X)\)-modules. Then the product of \(\mathscr M_\star\) in \({\bf BanMod}_{\mathbb X}\) exists and is given by
\[
\prod_{i\in I}^{{\bf BanMod}_{\mathbb X}}\mathscr M_i\cong\ell_\infty(\mathscr M_\star),
\]
together with the morphisms \(\pi_i\colon\ell_\infty(\mathscr M_\star)\to\mathscr M_i\) defined as \(\pi_i(v_\star)\coloneqq v_i\)
for every \(v_\star\in\ell_\infty(\mathscr M_\star)\).
\end{theorem}
\begin{proof}
First of all, observe that each mapping \(\pi_i\) is a morphism in \({\bf BanMod}_{\mathbb X}\). Now fix a Banach
\(L^0(\mathbb X)\)-module \(\mathscr N\) and a family \(\{\varphi_i\colon\mathscr N\to\mathscr M_i\}_{i\in I}\) of morphisms.
Let us define \(\Phi\colon\mathscr N\to\ell_\infty(\mathscr M_\star)\) as \(\Phi(w)\coloneqq\big(\varphi_i(w)\big)_{i\in I}\)
for every \(w\in\mathscr N\). Notice that \(\Phi\) is the unique mapping from \(\mathscr N\) to \(\ell_\infty(\mathscr M_\star)\)
satisfying \(\pi_i\circ\Phi=\varphi_i\) for every \(i\in I\). Clearly, \(\Phi\) is a morphism of \(L^0(\mathbb X)\)-modules.
Moreover, by passing to the supremum over \(i\in I\), we deduce from \(|\varphi_i(w)|\leq|w|\) that \(|\Phi(w)|_\infty\leq|w|\)
holds for every \(w\in\mathscr N\). All in all, \(\Phi\) is a morphism of Banach \(L^0(\mathbb X)\)-modules. The proof is complete.
\end{proof}
\begin{theorem}[Coproducts in \({\bf BanMod}_{\mathbb X}\)]\label{thm:coprod_BanMod}
Let \(\mathbb X\) be a \(\sigma\)-finite measure space, \(\mathscr M_\star=\{\mathscr M_i\}_{i\in I}\) a set
of Banach \(L^0(\mathbb X)\)-modules. Then the coproduct of \(\mathscr M_\star\) in \({\bf BanMod}_{\mathbb X}\) exists and is given by
\[
\coprod_{i\in I}^{{\bf BanMod}_{\mathbb X}}\mathscr M_i\cong\ell_1(\mathscr M_\star),
\]
together with the morphisms \(\iota_i\colon\mathscr M_i\to\ell_1(\mathscr M_\star)\) defined as
\(\iota_i(v)\coloneqq(w_j)_{j\in I}\) for every \(v\in\mathscr M_i\),
where we set \(w_i\coloneqq v\) and \(w_j\coloneqq 0_{\mathscr M_j}\) for every \(j\in I\setminus\{i\}\).
\end{theorem}
\begin{proof}
First of all, observe that each mapping \(\iota_i\) is a morphism in \({\bf BanMod}_{\mathbb X}\) (with \(|\iota_i(v)|_1=|v|\)
for every \(v\in\mathscr M_i\)). Now fix a Banach \(L^0(\mathbb X)\)-module \(\mathscr N\) and a family
\(\{\varphi_i\colon\mathscr M_i\to\mathscr N\}_{i\in I}\) of morphisms. We claim that, given any element \(v_\star\in\ell_1(\mathscr M_\star)\),
the series \(\sum_{i\in I}\varphi_i(v_i)\) is unconditionally convergent in \(\mathscr N\) to some \(\Phi(v_\star)\in\mathscr N\).
To prove it, pick some partition \((E_n)_{n\in\N}\subseteq\Sigma\) of \(\X\) that satisfies \(\mm(E_n)<+\infty\) and \(\nchi_{E_n}|v_\star|_1\leq n\)
for every \(n\in\N\). For each \(J\in\mathcal P_F(I)\) we have that
\[
\sum_{i\in J}\int_{E_n}|\varphi_i(v_i)|\,\d\mm\leq\sum_{i\in J}\int_{E_n}|v_i|\,\d\mm\leq\int_{E_n}|v_\star|_1\,\d\mm\leq n\,\mm(E_n).
\]
It follows that \(\sum_{i\in I}\big\|\nchi_{E_n}|\varphi_i(v_i)|\big\|_{L^1(\mathbb X)}\leq n\,\mm(E_n)<+\infty\) for every \(n\in\N\),
thus in particular the series \(\sum_{i\in I}\nchi_{E_n}\cdot\varphi_i(v_i)\) is unconditionally convergent in \(\mathscr N\).
Recalling Remark \ref{rmk:about_conv_L0}, we conclude that the series \(\sum_{i\in I}\varphi_i(v_i)\) converges unconditionally to some
element \(\Phi(v_\star)\in\mathscr N\), as we claimed.

One can readily check that the resulting mapping \(\Phi\colon\ell_1(\mathscr M_\star)\to\mathscr N\) is a morphism of Banach
\(L^0(\mathbb X)\)-modules and that it satisfies \(\Phi\circ\iota_i=\varphi_i\) for every \(i\in I\). It only remains to show that
\(\Phi\) is the unique mapping having these two properties. To this aim, fix a morphism \(\Psi\colon\ell_1(\mathscr M_\star)\to\mathscr N\)
such that \(\Psi\circ\iota_i=\varphi_i\) for every \(i\in I\). Given any \(v_\star\in\ell_1(\mathscr M_\star)\), we have that the
series \(\sum_{i\in I}\iota_i(v_i)\) converges unconditionally to \(v_\star\) in \(\ell_1(\mathscr M_\star)\). Hence, the linearity and the
continuity of \(\Psi\), \(\Phi\) yield
\[
\Psi(v_\star)=\Psi\bigg(\sum_{i\in I}\iota_i(v_i)\bigg)=\sum_{i\in I}\Psi\big(\iota_i(v_i)\big)=\sum_{i\in I}\varphi_i(v_i)
=\sum_{i\in I}\Phi\big(\iota_i(v_i)\big)=\Phi\bigg(\sum_{i\in I}\iota_i(v_i)\bigg)=\Phi(v_\star),
\]
which proves the uniqueness of \(\Phi\). Consequently, the statement is achieved.
\end{proof}
\subsubsection*{The bicompleteness of \({\bf BanMod}_{\mathbb X}\)}
It is now immediate to obtain the main result of this paper:
\begin{theorem}\label{thm:bicompl}
Let \(\mathbb X\) be a \(\sigma\)-finite measure space. Then the category \({\bf BanMod}_{\mathbb X}\) is bicomplete.
\end{theorem}
\begin{proof}
It follows from Theorems \ref{thm:kernel_BanMod}, \ref{thm:cokernel_BanMod}, \ref{thm:prod_BanMod}, \ref{thm:coprod_BanMod},
and \ref{thm:suff_(co)complete}.
\end{proof}
\subsubsection*{Description of other limits and colimits in \texorpdfstring{\({\bf BanMod}_{\mathbb X}\)}{BanModX}}
Theorem \ref{thm:bicompl} ensures that inverse/direct limits and pullbacks/pushouts always exist in \({\bf BanMod}_{\mathbb X}\).
However, we believe it is also useful to describe them explicitly. We shall only write the relevant statements, omitting their proofs.
\begin{proposition}[Pullbacks in \({\bf BanMod}_{\mathbb X}\)]\label{prop:pullback_BanMod}
Let \(\mathbb X\) be a \(\sigma\)-finite measure space. Let \(\mathscr M\), \(\mathscr N\), and \(\mathscr Q\) be Banach \(L^0(\mathbb X)\)-modules.
Let \(\varphi\colon\mathscr M\to\mathscr Q\) and \(\psi\colon\mathscr N\to\mathscr Q\) be morphisms in
\({\bf BanMod}_{\mathbb X}\). Then the pullback of \(\varphi\) and \(\psi\) in \({\bf BanMod}_{\mathbb X}\) is given by
\[
\mathscr M\times_{\mathscr Q}\mathscr N\cong\big\{(v,w)\in\mathscr M\oplus_\infty\mathscr N\;\big|\;\varphi(v)=\psi(w)\big\}
\]
together with \(p_{\mathscr M}\coloneqq\pi_{\mathscr M}|_{\mathscr M\times_{\mathscr Q}\mathscr N}\colon\mathscr M\times_{\mathscr Q}\mathscr N\to\mathscr M\)
and \(p_{\mathscr N}\coloneqq\pi_{\mathscr N}|_{\mathscr M\times_{\mathscr Q}\mathscr N}\colon\mathscr M\times_{\mathscr Q}\mathscr N\to\mathscr N\), where
\((\mathscr M\oplus_\infty\mathscr N,\pi_{\mathscr M},\pi_{\mathscr N})\) stands for the product of \(\{\mathscr M,\mathscr N\}\) in the category \({\bf BanMod}_{\mathbb X}\).
\end{proposition}
\begin{proposition}[Pushouts in \({\bf BanMod}_{\mathbb X}\)]\label{prop:pushout_BanMod}
Let \(\mathbb X\) be a \(\sigma\)-finite measure space. Let \(\mathscr M\), \(\mathscr N\), and \(\mathscr Q\) be Banach \(L^0(\mathbb X)\)-modules.
Let \(\varphi\colon\mathscr Q\to\mathscr M\) and \(\psi\colon\mathscr Q\to\mathscr N\) be morphisms in
\({\bf BanMod}_{\mathbb X}\). Let us define the normed \(L^0(\mathbb X)\)-submodule \(\mathscr I\) of \(\mathscr M\oplus_1\mathscr N\) as
\[
\mathscr I\coloneqq\Big\{(v,w)\in\mathscr M\oplus_1\mathscr N\;\big|\;
\big(\varphi(z),\psi(z)\big)=(-v,w)\,\text{ for some }z\in\mathscr Q\Big\}.
\]
Then the pushout of \(\varphi\) and \(\psi\) in \({\bf BanMod}_{\mathbb X}\) is given by
\[
\mathscr M\sqcup_{\mathscr Q}\mathscr N\cong(\mathscr M\oplus_1\mathscr N)/{\rm cl}_{\mathscr M\oplus_1\mathscr N}(\mathscr I)
\]
together with the morphisms \(i_{\mathscr M}\coloneqq\pi\circ\iota_{\mathscr M}\colon\mathscr M\to\mathscr M\sqcup_{\mathscr Q}\mathscr N\)
and \(i_{\mathscr N}\coloneqq\pi\circ\iota_{\mathscr N}\colon\mathscr N\to\mathscr M\sqcup_{\mathscr Q}\mathscr N\), where
\((\mathscr M\oplus_1\mathscr N,\iota_{\mathscr M},\iota_{\mathscr N})\) stands for the coproduct of \(\{\mathscr M,\mathscr N\}\)
in the category \({\bf BanMod}_{\mathbb X}\), while by \(\pi\colon\mathscr M\oplus_1\mathscr N\to\mathscr M\sqcup_{\mathscr Q}\mathscr N\)
we mean the canonical projection map on the quotient.
\end{proposition}

Furthermore, by combining Proposition \ref{prop:pushout_BanMod} with Theorem \ref{thm:kernel_BanMod} and, respectively,
Proposition \ref{prop:pullback_BanMod} with Theorem \ref{thm:cokernel_BanMod}, one obtains the following two results:
\begin{corollary}[Images in \({\bf BanMod}_{\mathbb X}\)]
Let \(\mathbb X\) be a \(\sigma\)-finite measure space and \(\varphi\colon\mathscr M\to\mathscr N\) a morphism in \({\bf BanMod}_{\mathbb X}\).
Then there exists a unique morphism \(\psi\colon\mathscr M/{\rm Ker}(\varphi)\to\mathscr N\) in \({\bf BanMod}_{\mathbb X}\) such that
\(\psi\big(v+{\rm Ker}(\varphi)\big)=\varphi(v)\) holds for every \(v\in\mathscr M\). Moreover, it holds that
\[
\big({\rm Im}(\varphi),{\rm im}(\varphi)\big)\cong\big(\mathscr M/{\rm Ker}(\varphi),\psi\big).
\]
\end{corollary}
\begin{corollary}[Coimages in \({\bf BanMod}_{\mathbb X}\)]
Let \(\mathbb X\) be a \(\sigma\)-finite measure space and \(\varphi\colon\mathscr M\to\mathscr N\) a morphism in \({\bf BanMod}_{\mathbb X}\).
Then it holds that
\[
\big({\rm Coim}(\varphi),{\rm coim}(\varphi)\big)\cong\big({\rm cl}_{\mathscr N}(\varphi(\mathscr M)),\varphi\big).
\]
\end{corollary}

Finally, we provide an explicit description of inverse and direct limits in \({\bf BanMod}_{\mathbb X}\). Recall that inverse
and direct limits always exist in the category \(R\text{-}{\bf Mod}\) of modules over a commutative ring \(R\), see e.g.\ \cite{lang84}.
We shall denote by \({\sf f}_{\mathbb X}\colon{\bf BanMod}_{\mathbb X}\to L^0(\mathbb X)\text{-}{\bf Mod}\) the forgetful functor.
\begin{proposition}[Inverse limits in \({\bf BanMod}_{\mathbb X}\)]\label{prop:inverse_lim_BanModX}
Let \((I,\leq)\) be a directed set. Let \(\mathbb X\) be a \(\sigma\)-finite measure space. Let \(\big(\{\mathscr M_i\}_{i\in I},\{{\rm P}_{ij}\}_{i\leq j}\big)\)
be an inverse system in the category \({\bf BanMod}_{\mathbb X}\). We denote by \((M,\{\tilde{\rm P}_i\}_{i\in I})\) the inverse limit of
\(\big(\{{\sf f}_{\mathbb X}(\mathscr M_i)\}_{i\in I},\{{\sf f}_{\mathbb X}({\rm P}_{ij})\}_{i\leq j}\big)\) in \(L^0(\mathbb X)\text{-}{\bf Mod}\).
Let us define the mapping \(|\cdot|\colon M\to L^0_{\rm ext}(\mathbb X)\) as
\[
|v|\coloneqq\bigvee_{i\in I}\big|\tilde{\rm P}_i(v)\big|\quad\text{ for every }v\in M.
\]
Then \(\tilde M\coloneqq\big\{v\in M\,:\,|v|\in L^0(\mathbb X)\big\}\) is an \(L^0(\mathbb X)\)-submodule of \(M\)
and \(|\cdot|\) is a complete pointwise norm on \(\tilde M\). Moreover, the inverse limit of
\(\big(\{\mathscr M_i\}_{i\in I},\{{\rm P}_{ij}\}_{i\leq j}\big)\) in \({\bf BanMod}_{\mathbb X}\) is given by
\[
\varprojlim\mathscr M_\star\cong\tilde M
\]
together with the morphisms \({\rm P}_i\coloneqq\tilde{\rm P}_i|_{\tilde M}\colon\tilde M\to\mathscr M_i\).
\end{proposition}
\begin{proposition}[Direct limits in \({\bf BanMod}_{\mathbb X}\)]\label{prop:direct_lim_BanModX}
Let \((I,\leq)\) be a directed set and \(\mathbb X\) a \(\sigma\)-finite measure space. Let \(\big(\{\mathscr M_i\}_{i\in I},\{\varphi_{ij}\}_{i\leq j}\big)\)
be a direct system in the category \({\bf BanMod}_{\mathbb X}\). Let us denote by \((M,\{\tilde\varphi_i\}_{i\in I})\) the direct limit of
\(\big(\{{\sf f}_{\mathbb X}(\mathscr M_i)\}_{i\in I},\{{\sf f}_{\mathbb X}(\varphi_{ij})\}_{i\leq j}\big)\) in \(L^0(\mathbb X)\text{-}{\bf Mod}\). Then
\[
|w|\coloneqq\bigwedge\big\{|v|\;\big|\;i\in I,\,v\in\mathscr M_i,\,\tilde\varphi_i(v)=w\big\}\quad\text{ for every }w\in M
\]
defines a pointwise seminorm \(|\cdot|\colon M\to L^0(\mathbb X)\). Moreover, the direct limit of \(\big(\{\mathscr M_i\}_{i\in I},\{\varphi_{ij}\}_{i\leq j}\big)\)
in \({\bf BanMod}_{\mathbb X}\) is given by
\[
\varinjlim\mathscr M_\star\cong\mathscr M\coloneqq\overline{M/\sim_{|\cdot|}}
\]
together with the morphisms \(\varphi_i\coloneqq\iota\circ\pi\circ\tilde\varphi_i\colon\mathscr M_i\to\mathscr M\),
where \(\pi\colon M\to M/\sim_{|\cdot|}\) is the canonical projection on the quotient, while \((\mathscr M,\iota)\)
is the completion of the normed \(L^0(\mathbb X)\)-module \(M/\sim_{|\cdot|}\).
\end{proposition}

Next we provide an elementary example of a nontrivial inverse system of Banach \(L^0(\mathbb X)\)-modules having a trivial inverse limit.
This construction will be useful in Example \ref{ex:inverse_not_right_exact}.
\begin{example}\label{ex:inverse_trivial}{\rm
Let \(\mathbb X\) be a \(\sigma\)-finite measure space and \(\mathscr M\neq\{0\}\) a given Banach \(L^0(\mathbb X)\)-module. We define \(\mathscr M_n\coloneqq\mathscr M\)
for every \(n\in\N\). Given any \(n,m\in\N\) with \(n\leq m\), we define the morphism \({\rm P}_{nm}\colon\mathscr M_m\to\mathscr M_n\) as
\({\rm P}_{nm}\coloneqq\frac{n}{m}{\rm id}_{\mathscr M}\). Then \(\big(\{\mathscr M_n\}_{n\in\N},\{{\rm P}_{nm}\}_{n\leq m}\big)\) is an inverse system and
\begin{equation}\label{eq:inverse_trivial}
\textstyle\varprojlim\mathscr M_\star\cong\{0\}.
\end{equation}
Indeed, the inverse limit \((M,\{\tilde{\rm P}_n\}_{n\in\N})\) of \(\big(\{{\sf f}_{\mathbb X}(\mathscr M_n)\}_{n\in\N},\{{\sf f}_{\mathbb X}({\rm P}_{nm})\}_{n\leq m}\big)\)
in \(L^0(\mathbb X)\text{-}{\bf Mod}\) is given by \(M\coloneqq\big\{(kv)_{k\in\N}\in\prod_{k\in\N}^{\bf Set}\mathscr M_k\,:\,v\in\mathscr M\big\}\) together with the morphisms
\(\tilde{\rm P}_n\colon M\to\mathscr M_n\) defined as
\[
\tilde{\rm P}_n\big((kv)_{k\in\N}\big)\coloneqq nv\quad\text{ for every }(kv)_{k\in\N}\in M.
\]
Using Proposition \ref{prop:inverse_lim_BanModX}, we get
\(\big|(kv)_{k\in\N}\big|=\bigvee_{n\in\N}|nv|=(+\infty)\cdot\nchi_{\{|v|>0\}}\), which gives \eqref{eq:inverse_trivial}.
\fr}\end{example}

We know that the inverse limit functor preserves limits, whereas the direct limit functor preserves colimits.
On the contrary, the following two examples show that in \({\bf BanMod}_{\mathbb X}\) the inverse limit
functor is not right exact and that the direct limit functor is not left exact, respectively.
\begin{example}\label{ex:inverse_not_right_exact}{\rm
Fix any \(\sigma\)-finite measure space \(\mathbb X\) and any Banach \(L^0(\mathbb X)\)-module \(\mathscr M\neq\{0\}\). Given any \(n\in\N\), we define
\(\mathscr M_n=\mathscr N_n\coloneqq\mathscr M\). Given any \(n,m\in\N\) such that \(n\leq m\), we define \({\rm P}_{nm}\colon\mathscr M_m\to\mathscr M_n\)
and \({\rm Q}_{nm}\colon\mathscr N_m\to\mathscr N_n\) as \({\rm P}_{nm}\coloneqq\frac{n}{m}{\rm id}_{\mathscr M}\) and \({\rm Q}_{nm}\coloneqq{\rm id}_{\mathscr M}\), respectively.
Then both \(\big(\{\mathscr M_n\}_{n\in\N},\{{\rm P}_{nm}\}_{n\leq m}\big)\) and \(\big(\{\mathscr N_n\}_{n\in\N},\{{\rm Q}_{nm}\}_{n\leq m}\big)\) are inverse systems in
\({\bf BanMod}_{\mathbb X}\) and the collection \(\theta_\star\) of morphisms \(\theta_n\colon\mathscr M_n\to\mathscr N_n\) given by \(\theta_n\coloneqq\frac{1}{n}{\rm id}_{\mathscr M}\)
is a natural transformation, i.e.\ a morphism in \(({\bf BanMod}_{\mathbb X})^{(I,\leq)^{\rm op}}\). Observe that trivially \(\varprojlim_\N\mathscr N_\star\cong\mathscr M\), while
\(\varprojlim_\N\mathscr M_\star\cong\{0\}\) thanks to Example \ref{ex:inverse_trivial}. Recalling Remark \ref{rmk:kernel_of_natural_hom} and using the surjectivity of each \(\theta_n\),
we obtain that \({\rm Coker}(\theta_\star)(n)=\{0\}\) for all \(n\in\N\). Hence, \({\rm Coker}(\theta_\star)\) is the zero object of
\(({\bf BanMod}_{\mathbb X})^{(I,\leq)^{\rm op}}\) and thus \(\varprojlim_\N{\rm Coker}(\theta_\star)\cong\{0\}\). On the other hand, the cokernel of
the morphism \(\varprojlim_\N\theta_\star\) is
\[
\textstyle
{\rm Coker}\big(\varprojlim_\N\theta_\star\big)\cong\big(\varprojlim_\N\mathscr N_\star\big)/{\rm cl}_{(\varprojlim_\N\mathscr N_\star)}\Big(\big(\varprojlim_\N\theta_\star\big)\big(\varprojlim_\N\mathscr M_\star\big)\Big)
\cong\varprojlim_\N\mathscr N_\star\cong\mathscr M\neq\{0\},
\]
thus the inverse limit functor \(\varprojlim_\N\) on \(({\bf BanMod}_{\mathbb X})^{(I,\leq)^{\rm op}}\) does not preserve cokernels.
\fr}\end{example}
\begin{example}\label{ex:direct_not_left_exact}{\rm
Fix any \(\sigma\)-finite measure space \(\mathbb X\). Consider the morphism \(\theta\colon\mathscr H_{\mathbb X}(\N)\to\mathscr H_{\mathbb X}(\N)\) defined as follows:
given any \(v\in\mathscr H_{\mathbb X}(\N)\), we set \(\theta(v)(1)\coloneqq 0\) and \(\theta(v)(i)\coloneqq v(i)\) for every \(i\in\N\) with \(i\geq 2\).
We define the sequence \((v_n)_{n\in\N}\subseteq\mathscr H_{\mathbb X}(\N)\) as \(v_1(i)\coloneqq\frac{1}{i}\nchi_\X\) for every \(i\in\N\) and
\(v_n\coloneqq e_n\) for every \(n\geq 2\). Moreover, we denote by \(\mathscr M_n\) the closure in \(\mathscr H_{\mathbb X}(\N)\) of its \(L^0(\mathbb X)\)-submodule
generated by \(\{v_1,\ldots,v_n\}\), while we define \(\mathscr N_n\coloneqq\mathscr H_{\mathbb X}(\N)\). Finally, for any \(n\leq m\) we denote by
\(\varphi_{nm}\colon\mathscr M_n\hookrightarrow\mathscr M_m\) the inclusion map, by \(\psi_{nm}\colon\mathscr N_n\to\mathscr N_m\) the identity map,
and by \(\theta_n\colon\mathscr M_n\to\mathscr N_n\) the map \(\theta_n\coloneqq\theta|_{\mathscr M_n}\). It holds that \(\big(\{\mathscr M_n\}_{n\in\N},\{\varphi_{nm}\}_{n\leq m}\big)\)
and \(\big(\{\mathscr N_n\}_{n\in\N},\{\psi_{nm}\}_{n\leq m}\big)\) are direct systems in \({\bf BanMod}_{\mathbb X}\), while \(\theta_\star\) is a morphism
in \(({\bf BanMod}_{\mathbb X})^{(\N,\leq)}\) between them. Recalling Remark \ref{rmk:kernel_of_natural_hom}, we deduce that \({\rm Ker}(\theta_\star)\) is
the zero object, so that \(\varinjlim_\N{\rm Ker}(\theta_\star)=\{0\}\). On the other hand, since the \(L^0(\mathbb X)\)-submodule of \(\mathscr H_{\mathbb X}(\N)\)
generated by the sequence \((v_n)_{n\in\N}\) is dense in \(\mathscr H_{\mathbb X}(\N)\), one can easily check that \(\varinjlim_\N\mathscr M_\star\cong\mathscr H_{\mathbb X}(\N)\).
Clearly, \(\varinjlim_\N\mathscr N_\star\cong\mathscr H_{\mathbb X}(\N)\) as well and \(\varinjlim_\N\theta_\star\cong\theta\). Since \(\theta\) is not injective,
we conclude that \({\rm Ker}\big(\varinjlim_\N\theta_\star\big)\neq\{0\}\), thus showing that the direct limit functor \(\varinjlim_\N\) on \(({\bf BanMod}_{\mathbb X})^{(\N,\leq)}\) does not preserve kernels.
\fr}\end{example}
\subsection{The inverse image functor}\label{ss:inv_img}
Here, we introduce and study the `inverse image functor'.
\subsubsection*{The category \({\bf BanMod}\)}
Let us now consider the category \({\bf BanMod}\), which is defined as follows:
\begin{itemize}
\item The objects of \({\bf BanMod}\) are given by the couples \((\mathbb X,\mathscr M)\), where \(\mathbb X\) is a \(\sigma\)-finite measure space
and \(\mathscr M\) is a Banach \(L^0(\mathbb X)\)-module.
\item A morphism in \({\bf BanMod}\) between two objects \((\mathbb X,\mathscr M)\) and \((\mathbb Y,\mathscr N)\) is a couple \((\tau,\varphi)\),
where \(\tau\colon\mathbb X\to\mathbb Y\) is a morphism in \({\bf Meas}_\sigma\) and \(\varphi\colon\mathscr N\to\mathscr M\) is a linear map such that
\[\begin{split}
\varphi(f\cdot v)=(f\circ\tau)\cdot\varphi(v)&\quad\text{ for every }f\in L^0(\mathbb Y)\text{ and }v\in\mathscr N,\\
|\varphi(v)|\leq|v|\circ\tau&\quad\text{ for every }v\in\mathscr N.
\end{split}\]
\item Given morphisms \((\tau,\varphi)\colon(\mathbb X,\mathscr M)\to(\mathbb Y,\mathscr N)\) and \((\eta,\psi)\colon(\mathbb Y,\mathscr N)\to(\mathbb W,\mathscr Q)\)
in \({\bf BanMod}\), their composition is defined as \((\eta,\psi)\circ(\tau,\varphi)\coloneqq(\eta\circ\tau,\varphi\circ\psi)\colon(\mathbb X,\mathscr M)\to(\mathbb W,\mathscr Q)\).
\end{itemize}
It holds that \({\bf Meas}_\sigma\) and \({\bf Ban}^{\rm op}\) can be realised as full subcategories of \({\bf BanMod}\), while each \({\bf BanMod}_{\mathbb X}^{\rm op}\)
can be realised as a (not necessarily full) subcategory of \({\bf BanMod}\). More precisely:
\begin{itemize}
\item Define the functor \({\rm I}_M\colon{\bf Meas}_\sigma\to{\bf BanMod}\) as \({\rm I}_M(\mathbb X)\coloneqq(\mathbb X,\{0\})\) for every object \(\mathbb X\) of
\({\bf Meas}_\sigma\) and \({\rm I}_M(\tau)\coloneqq(\tau,0)\colon(\mathbb X,\{0\})\to(\mathbb Y,\{0\})\) for every morphism \(\tau\colon\mathbb X\to\mathbb Y\)
in \({\bf Meas}_\sigma\). Then \({\rm I}_M\) is a fully faithful functor that is injective on objects.
\item Let \(\mathbb X\) be a fixed \(\sigma\)-finite measure space. Define the functor \({\rm I}_{\mathbb X}\colon{\bf BanMod}_{\mathbb X}^{\rm op}\to{\bf BanMod}\)
as \({\rm I}_{\mathbb X}(\mathscr M)\coloneqq(\mathbb X,\mathscr M)\) for every object \(\mathscr M\) of \({\bf BanMod}_{\mathbb X}\) and
\({\rm I}_{\mathbb X}(\varphi)\coloneqq({\rm id}_{\mathbb X},\varphi)\) for every morphism \(\varphi\colon\mathscr M\to\mathscr N\) in \({\bf BanMod}_{\mathbb X}^{\rm op}\).
Then \({\rm I}_{\mathbb X}\) is faithful and injective on objects.
\item Recall that \({\bf BanMod}_{\mathbb P}={\bf Ban}\). Also,
\({\rm id}_{\mathbb P}\colon\mathbb P\to\mathbb P\) is the unique element of
\({\rm Hom}_{{\bf Meas}_\sigma}(\mathbb P,\mathbb P)\), thus \({\rm I}_{\mathbb P}\)
is a full functor and \({\bf Ban}^{\rm op}\) can be realised as a full subcategory of \({\bf BanMod}\).
\end{itemize}
We shall also consider the forgetful functor \(\Pi_M\colon{\bf BanMod}\to{\bf Meas}_\sigma\), which we define as 
\[\begin{split}
\Pi_M((\mathbb X,\mathscr M))\coloneqq\mathbb X&\quad\text{ for every object }(\mathbb X,\mathscr M)\text{ of }{\bf BanMod},\\
\Pi_M((\tau,\varphi))\coloneqq\tau&\quad\text{ for every morphism }(\tau,\varphi)\colon(\mathbb X,\mathscr M)\to(\mathbb Y,\mathscr N)\text{ in }{\bf BanMod}.
\end{split}\]
\subsubsection*{Inverse image versus pullback}
We are now in a position to introduce the inverse image functor:
\begin{theorem}[Inverse image functor]\label{thm:inv_im}
There exists a unique functor
\[
{\rm InvIm}\colon({\rm id}_{{\bf Meas}_\sigma}\downarrow\Pi_M)\to{\bf BanMod}^\to,
\]
which we call the \emph{inverse image functor}, such that the following properties are satisfied:
\begin{itemize}
\item If \((\mathbb X,(\mathbb Y,\mathscr M),\tau)\) is a given object of the comma category \(({\rm id}_{{\bf Meas}_\sigma}\downarrow\Pi_M)\),
then there exist a Banach \(L^0(\mathbb X)\)-module \(\tau^*\mathscr M\) and a linear operator \(\tau^*\colon\mathscr M\to\tau^*\mathscr M\) such that
\[
{\rm InvIm}\big((\mathbb X,(\mathbb Y,\mathscr M),\tau)\big)=(\tau,\tau^*).
\]
Moreover, the \(L^0(\mathbb X)\)-submodule of \(\tau^*\mathscr M\) generated by the range \(\tau^*(\mathscr M)\) is dense in \(\tau^*\mathscr M\)
and it holds that \(|\tau^*v|=|v|\circ\tau\) for every \(v\in\mathscr M\).
\item If \((\eta_1,(\eta_2,\varphi))\colon(\mathbb X_1,(\mathbb Y_1,\mathscr M_1),\tau_1)\to(\mathbb X_2,(\mathbb Y_2,\mathscr M_2),\tau_2)\) is a
morphism in \(({\rm id}_{{\bf Meas}_\sigma}\downarrow\Pi_M)\), then there exists an operator \(\psi\colon\tau_2^*\mathscr M_2\to\tau_1^*\mathscr M_1\)
such that
\[
{\rm InvIm}\big((\eta_1,(\eta_2,\varphi))\big)=\big((\eta_1,\psi),(\eta_2,\varphi)\big).
\]
\end{itemize}
The uniqueness of the functor \({\rm InvIm}\) is intended up to a unique natural isomorphism in the functor category from
\(({\rm id}_{{\bf Meas}_\sigma}\downarrow\Pi_M)\) to \({\bf BanMod}^\to\).
\end{theorem}

The statement of Theorem \ref{thm:inv_im} is a reformulation of various results about `pullback modules' contained in \cite[Section 1.6]{Gigli14},
or rather of their corresponding versions for Banach \(L^0\)-modules, which can be obtained by suitably adapting the same proof arguments. Let us briefly
comment on the terminology: in \cite{Gigli14} the Banach module \(\tau^*\mathscr M\) is called the \emph{pullback module} of \(\mathscr M\)
with respect to \(\tau\), by analogy with the notion of pullback that is commonly used in differential geometry. Since in this paper we are studying these
topics from the perspective of category theory, we prefer to adopt the term `inverse image', in order to avoid confusion with the categorical
notion of pullback. Nevertheless, the two concepts are strongly related, as observed in \cite[Remark 1.6.4]{Gigli14}. Namely:
\begin{theorem}
Let \(\tau\colon\mathbb X\to\mathbb Y\) be a morphism in \({\bf Meas}_\sigma\) and \(\mathscr M\) a Banach \(L^0(\mathbb Y)\)-module.
Then the pullback of \(({\rm id}_{\mathbb Y},0)\colon(\mathbb Y,\mathscr M)\to(\mathbb Y,\{0\})\) and
\((\tau,0)\colon(\mathbb X,\{0\})\to(\mathbb Y,\{0\})\) exists in \({\bf BanMod}\):
\[
(\mathbb Y,\mathscr M)\times_{(\mathbb Y,\{0\})}(\mathbb X,\{0\})\cong(\mathbb X,\tau^*\mathscr M)
\]
together with the morphisms \((\tau,\tau^*)\colon(\mathbb X,\tau^*\mathscr M)\to(\mathbb Y,\mathscr M)\)
and \(({\rm id}_{\mathbb X},0)\colon(\mathbb X,\tau^*\mathscr M)\to(\mathbb X,\{0\})\).
\end{theorem}

Given a fixed morphism \(\tau\colon\mathbb X\to\mathbb Y\) in \({\bf Meas}_\sigma\), the inverse image functor induces a functor
\[
{\rm InvIm}_\tau\colon{\bf BanMod}_{\mathbb Y}\to{\bf BanMod}_{\mathbb X}
\]
in the following way: for any Banach \(L^0(\mathbb Y)\)-module \(\mathscr M\) we define \({\rm InvIm}_\tau(\mathscr M)\coloneqq\tau^*\mathscr M\) and
for any morphism \(\varphi\colon\mathscr M\to\mathscr N\) in \({\bf BanMod}_{\mathbb Y}\) we define \({\rm InvIm}_\tau(\varphi)\coloneqq\tau^*\varphi\),
where by \(\tau^*\varphi\colon\tau^*\mathscr M\to\tau^*\mathscr N\) we mean the unique morphism in \({\bf BanMod}_{\mathbb X}\) with
\({\rm InvIm}\big(({\rm id}_{\mathbb X},({\rm id}_{\mathbb Y},\varphi)\big)=\big(({\rm id}_{\mathbb X},\tau^*\varphi),({\rm id}_{\mathbb Y},\varphi)\big)\).
\begin{example}\label{ex:LB_as_pullback}{\rm
Let \(\mathbb X\) be a finite measure space. Let \(\mathscr M\) be a given Banach \(L^0(\mathbb Y)\)-module,
for some \(\sigma\)-finite measure space \(\mathbb Y\). We define \(\pi_\Y\colon\X\times\Y\to\Y\)
as \(\pi_\Y(x,y)\coloneqq y\) for every \((x,y)\in\X\times\Y\). Clearly, \(\pi_\Y\) is measurable and
\((\pi_\Y)_\#(\mm_\X\otimes\mm_\Y)=\mm_\X(\X)\mm_\Y\), so that \(\pi_\Y\colon\mathbb X\times\mathbb Y\to\mathbb Y\)
is a morphism in \({\bf Meas}_\sigma\). We define \(c\colon\mathscr M\to L^0(\mathbb X;\mathscr M)\)
as \(c(v)\coloneqq\nchi_\X v\) for every \(v\in\mathscr M\). Then we claim that
\begin{equation}\label{eq:LB_as_pullback}
\big(L^0(\mathbb X;\mathscr M),c\big)\cong(\pi_\Y^*\mathscr M,\pi_\Y^*).
\end{equation}
Its validity follows from the linearity of \(c\), from the fact that for any given \(v\in\mathscr M\) the identities
\[
|c(v)|(x,y)=\big|c(v)(x)\big|(y)=|v|(y)=|v|\big(\pi_\Y(x,y)\big)=(|v|\circ\pi_\Y)(x,y)
\]
are verified for \((\mm_\X\otimes\mm_\Y)\)-a.e.\ \((x,y)\), and from the density of simple maps in \(L^0(\mathbb X;\mathscr M)\).
\fr}\end{example}

In the following result, we check (by a direct verification) that \({\rm InvIm}_\tau\) preserves direct limits:
\begin{proposition}
Let \(\tau\colon\mathbb X\to\mathbb Y\) be a morphism in \({\bf Meas}_\sigma\). Let \(\big(\{\mathscr M_i\}_{i\in I},\{\varphi_{ij}\}_{i\leq j}\}\big)\)
be a direct system in \({\bf BanMod}_{\mathbb Y}\), with direct limit \(\big(\varinjlim\mathscr M_\star,\{\varphi_i\}_{i\in I}\big)\).
Then \(\big(\{\tau^*\mathscr M_i\}_{i\in I},\{\tau^*\varphi_{ij}\}_{i\leq j}\big)\) is a direct system in \({\bf BanMod}_{\mathbb X}\),
whose direct limit is given by
\begin{equation}\label{eq:dir_lim_and_inv_im}
\varinjlim\tau^*\mathscr M_\star\cong\tau^*\varinjlim\mathscr M_\star
\end{equation}
together with the family \(\{\tau^*\varphi_i\}_{i\in I}\) of morphisms \(\tau^*\varphi_i\colon\tau^*\mathscr M_i\to\tau^*\varinjlim\mathscr M_\star\).
\end{proposition}
\begin{proof}
Clearly, we have \((\tau^*\varphi_j)\circ(\tau^*\varphi_{ij})=\tau^*\varphi_i\) for all \(i,j\in I\) with \(i\leq j\). Fix any
\(\big(\mathscr N,\{\psi_i\}_{i\in I}\big)\), where \(\mathscr N\) is a Banach \(L^0(\mathbb X)\)-module and the morphisms
\(\psi_i\colon\tau^*\mathscr M_i\to\mathscr N\) satisfy \(\psi_j\circ(\tau^*\varphi_{ij})=\psi_i\) for every \(i\leq j\).
We claim that there exists a unique morphism \(\Phi\colon\tau^*\varinjlim\mathscr M_\star\to\mathscr N\) such that
\begin{equation}\label{eq:char_Phi}
\Phi\big(\tau^*(\varphi_i(v))\big)=\psi_i(\tau^*v)\quad\text{ for every }i\in I\text{ and }v\in\mathscr M_i,
\end{equation}
whence the statement follows. Given that \(\bigcup_{i\in I}\varphi_i(\mathscr M_i)\) is a dense \(L^0(\mathbb Y)\)-submodule of
\(\varinjlim\mathscr M_\star\) by Proposition \ref{prop:direct_lim_BanModX}, the \(L^0(\mathbb X)\)-submodule
of \(\tau^*\varinjlim\mathscr M_\star\) generated by \(\bigcup_{i\in I}\big\{\tau^*(\varphi_i(v))\,:\,v\in\mathscr M_i\big\}\)
is dense in \(\tau^*\varinjlim\mathscr M_\star\), which forces the uniqueness of the morphism \(\Phi\).
We now pass to the verification of (the well-posedness and of) the existence of \(\Phi\). We denote by \((M,\{\tilde\varphi_i\}_{i\in I})\) the direct limit of
\(\big(\{{\sf f}_{\mathbb Y}(\mathscr M_i)\}_{i\in I},\{{\sf f}_{\mathbb Y}(\varphi_{ij})\}_{i\leq j}\big)\) in \(L^0(\mathbb Y)\text{-}{\bf Mod}\).
Observe that the map \(\tilde\Psi\colon M\to\mathscr N\), which we define as \(\tilde\Psi\big(\tilde\varphi_i(v)\big)\coloneqq\psi_i(\tau^*v)\) for every \(i\in I\)
and \(v\in\mathscr M_i\), is well-posed and linear. Indeed, if \(z\in M\) can be written as \(z=\tilde\varphi_i(v)=\tilde\varphi_j(w)\), then
(recalling how direct limits in \(L^0(\mathbb Y)\text{-}{\bf Mod}\) are constructed) we have \(\varphi_{ik}(v)=\varphi_{jk}(w)\) for some \(k\in I\) with \(i,j\leq k\),
so that the commutativity of the diagram
\[\begin{tikzcd}
\mathscr M_i \arrow[d,swap,"\tau^*"] \arrow[r,"\varphi_{ik}"] & \mathscr M_k \arrow[d,swap,"\tau^*"] & \mathscr M_j \arrow[l,swap,"\varphi_{jk}"] \arrow[d,"\tau^*"] \\
\tau^*\mathscr M_i \arrow[dr,swap,"\psi_i"] \arrow[r,"\tau^*\varphi_{ik}"] & \tau^*\mathscr M_k \arrow[d,swap,"\psi_k"] & \tau^*\mathscr M_j \arrow[l,swap,"\tau^*\varphi_{jk}"]
\arrow[dl,"\psi_j"] \\
& \mathscr N &
\end{tikzcd}\]
implies that \(\psi_i(\tau^*v)=\psi_j(\tau^*w)\), thus showing that \(\tilde\Psi(z)\) is well-posed. The linearity of \(\tilde\Psi\) follows by construction. Moreover, we can estimate
\(|\tilde\Psi(z)|\leq|v|\circ\tau\) for every \(z\in M\), \(i\in I\), and \(v\in\mathscr M_i\) with \(\tilde\varphi_i(v)=z\), whence it follows that \(|\tilde\Psi(z)|\leq|z|\circ\tau\)
for every \(z\in M\). It is then easy to check that there is a unique linear map \(\Psi\colon\overline{M/\sim_{|\cdot|}}\to\mathscr N\) satisfying
\(\Psi\big((\iota\circ\pi)(z)\big)=\tilde\Psi(z)\) for every \(z\in M\), where \(\pi\colon M\to M/\sim_{|\cdot|}\) is the projection and \(\big(\overline{M/\sim_{|\cdot|}},\iota\big)\)
is the completion of \(M/\sim_{|\cdot|}\). Notice also that \(|\Psi(z)|\leq|z|\circ\tau\) for every \(z\in\overline{M/\sim_{|\cdot|}}\). Recalling from Proposition \ref{prop:direct_lim_BanModX}
that \(\varinjlim\mathscr M_\star\cong\overline{M/\sim_{|\cdot|}}\), we conclude that there exists a unique morphism \(\Phi\colon\tau^*\varinjlim\mathscr M_\star\to\mathscr N\) for which the diagram
\[\begin{tikzcd}
\varinjlim\mathscr M_\star \arrow[d,swap,"\tau^*"] \arrow[r,"\Psi"] & \mathscr N \\
\tau^*\varinjlim\mathscr M_\star \arrow[ur,swap,"\Phi"] &
\end{tikzcd}\]
commutes, which is equivalent to requiring the validity of \eqref{eq:char_Phi}. This proves the statement.
\end{proof}

It also holds that if \(\tau\colon\mathbb X\to\mathbb Y\) is a morphism in \({\bf Meas}_\sigma\) and
\(\big(\{\mathscr M_i\}_{i\in I},\{{\rm P}_{ij}\}_{i\leq j}\big)\) is an inverse system in
\({\bf BanMod}_{\mathbb Y}\), then \(\big(\{\tau^*\mathscr M_i\}_{i\in I},\{\tau^*{\rm P}_{ij}\}_{i\leq j}\big)\)
is an inverse system in \({\bf BanMod}_{\mathbb X}\). Nevertheless, it can happen that
\(\varprojlim\tau^*\mathscr M_\star\ncong\tau^*\varprojlim\mathscr M_\star\). In other words, the functor
\({\rm InvIm}_\tau\) does not necessarily preserve inverse limits, as it is shown in Example \ref{ex:InvIm_no_inverse} below.
\section{The hom-functors in \texorpdfstring{\({\bf BanMod}_{\mathbb X}\)}{BanModX}}
Let \(\mathbb X\) be a given \(\sigma\)-finite measure space. Let \(\mathscr M\), \(\mathscr N\) be Banach \(L^0(\mathbb X)\)-modules. The functors
\[
\textsc{Hom}(\mathscr M,-)\colon{\bf BanMod}_{\mathbb X}\to{\bf BanMod}_{\mathbb X},\qquad
\textsc{Hom}(-,\mathscr N)\colon{\bf BanMod}_{\mathbb X}^{\rm op}\to{\bf BanMod}_{\mathbb X},
\]
which we call the \emph{hom-functors} in \({\bf BanMod}_{\mathbb X}\), are defined in the following way:
\begin{itemize}
\item Given an object \(\mathscr Q\) of \({\bf BanMod}_{\mathbb X}\), we set
\(\textsc{Hom}(\mathscr M,-)(\mathscr Q)\coloneqq\textsc{Hom}(\mathscr M,\mathscr Q)\).
Given a morphism \(\varphi\colon\mathscr Q\to\mathscr R\) in \({\bf BanMod}_{\mathbb X}\), we set
\(\textsc{Hom}(\mathscr M,-)(\varphi)\colon\textsc{Hom}(\mathscr M,\mathscr Q)\to\textsc{Hom}(\mathscr M,\mathscr R)\)
as \(\textsc{Hom}(\mathscr M,-)(\varphi)(T)\coloneqq\varphi\circ T\) for every \(T\in\textsc{Hom}(\mathscr M,\mathscr Q)\).
\item Given an object \(\mathscr Q\) of \({\bf BanMod}_{\mathbb X}\), we set
\(\textsc{Hom}(-,\mathscr N)(\mathscr Q)\coloneqq\textsc{Hom}(\mathscr Q,\mathscr N)\).
Given a morphism \(\varphi\colon\mathscr R\to\mathscr Q\) in \({\bf BanMod}_{\mathbb X}\), we set
\(\textsc{Hom}(-,\mathscr N)(\varphi^{\rm op})\colon\textsc{Hom}(\mathscr Q,\mathscr N)\to\textsc{Hom}(\mathscr R,\mathscr N)\)
as \(\textsc{Hom}(-,\mathscr N)(\varphi^{\rm op})(T)\coloneqq T\circ\varphi\) for every \(T\in\textsc{Hom}(\mathscr Q,\mathscr N)\).
\end{itemize}

In the following result, we obtain the expected continuity properties of the two hom-functors. We prove the statement directly,
rather than by applying a general principle, which would amount to showing that hom-functors are left/right adjoints to a
suitable notion of tensor product. The study of tensor products of Banach \(L^0\)-modules is committed to a future work.
\begin{proposition}[Continuity properties of the hom-functors]\label{prop:cont_hom-fct}
Let \(\mathbb X\) be a \(\sigma\)-finite measure space. Let \(\mathscr M\), \(\mathscr N\) be given Banach \(L^0(\mathbb X)\)-modules.
Then the following properties are verified:
\begin{itemize}
\item Let \(D\colon{\bf J}\to{\bf BanMod}_{\mathbb X}\) be a small diagram and denote by \((\mathscr L,\lambda_\star)\) its limit.
Then the limit of the diagram \(\tilde D\coloneqq\textsc{Hom}(\mathscr M,-)\circ D\colon{\bf J}\to{\bf BanMod}_{\mathbb X}\) is given by
\[
\big(\textsc{Hom}(\mathscr M,\mathscr L),\textsc{Hom}(\mathscr M,-)(\lambda_\star)\big).
\]
\item Let \(D\colon{\bf J}\to{\bf BanMod}_{\mathbb X}\) be a small diagram and denote by \((\mathscr C,c_\star)\) its colimit. We define the diagram
\(\tilde D\colon{\bf J}^{\rm op}\to{\bf BanMod}_{\mathbb X}\) as \(\tilde D(i)\coloneqq\textsc{Hom}(D(i),\mathscr N)\) for every \(i\in{\rm Ob}_{\bf J}\),
while we set \(\tilde D(\phi^{\rm op})\coloneqq\textsc{Hom}(-,\mathscr N)(D(\phi)^{\rm op})\) for every \(\phi\in{\rm Hom}_{\bf J}\). Then the limit of \(\tilde D\) is given by
\[
\big(\textsc{Hom}(\mathscr C,\mathscr N),\textsc{Hom}(-,\mathscr N)(c_\star^{\rm op})\big).
\]
\end{itemize}
\end{proposition}
\begin{proof}
We focus only on the first item, since the second one can by proved via similar arguments. For brevity, we denote \(\Phi_i\coloneqq\textsc{Hom}(\mathscr M,-)(\lambda_i)\)
for every \(i\in{\rm Ob}_{\bf J}\). Given a morphism \(\phi\colon i\to j\) in \({\bf J}\) and any \(T\in\textsc{Hom}(\mathscr M,\mathscr L)\), we have that
\((\tilde D(\phi)\circ\Phi_i)(T)=D(\phi)\circ\lambda_i\circ T=\lambda_j\circ T=\Phi_j(T)\), which shows that \(\big(\textsc{Hom}(\mathscr M,\mathscr L),\Phi_\star\big)\)
is a cone to \(\tilde D\). Now fix an arbitrary cone \((\mathscr Q,\Psi_\star)\) to \(\tilde D\). Given any \(z\in\mathscr Q\), we define
\(f_z\coloneqq\nchi_{\{|z|>0\}}\frac{1}{|z|}\in L^0(\mathbb X)\). Observe that for any \(i\in{\rm Ob}_{\bf J}\) the homomorphism \(T^z_i\coloneqq f_z\cdot\Psi_i(z)\colon\mathscr M\to D(i)\)
satisfies \(|T^z_i|\leq 1\), thus in particular it is a morphism in \({\bf BanMod}_{\mathbb X}\). Moreover, for any morphism \(\phi\colon i\to j\) in \({\bf J}\) and for
any \(v\in\mathscr M\) we can compute
\[
(D(\phi)\circ T^z_i)(v)=\big(D(\phi)\circ\Psi_i(z)\big)(f_z\cdot v)=(\tilde D(\phi)\circ\Psi_i)(z)(f_z\cdot v)=\Psi_j(z)(f_z\cdot v)=T^z_j(v),
\]
which shows that \((\mathscr M,T^z_\star)\) is a cone to \(D\). Hence, there exists a unique morphism \(\tilde\Phi(z)\colon\mathscr M\to\mathscr L\) such that
\(\lambda_i\circ\tilde\Phi(z)=T^z_i\) for every \(i\in{\rm Ob}_{\bf J}\). Letting \(\Phi(z)\coloneqq|z|\cdot\tilde\Phi(z)\in\textsc{Hom}(\mathscr M,\mathscr L)\)
for every \(z\in\mathscr Q\), we deduce that \(\Phi\colon\mathscr Q\to\textsc{Hom}(\mathscr M,\mathscr L)\) is the unique morphism in \({\bf BanMod}_{\mathbb X}\)
verifying the identity \(\Phi_i\circ\Phi=\Psi_i\) for every \(i\in{\rm Ob}_{\bf J}\). This proves that \(\big(\textsc{Hom}(\mathscr M,\mathscr L),\Phi_\star\big)\)
is the limit of \(\tilde D\).
\end{proof}

The next result immediately follows from Proposition \ref{prop:cont_hom-fct} (by just plugging \(\mathscr N=L^0(\mathbb X)\)):
\begin{corollary}
Let \(\mathbb X\) be a \(\sigma\)-finite measure space and \(\mathscr M\) a Banach \(L^0(\mathbb X)\)-module.
Fix a small diagram \(D\colon{\bf J}\to{\bf BanMod}_{\mathbb X}\), whose colimit we denote by \((\mathscr C,c_\star)\).
Define \(\tilde D\colon{\bf J}^{\rm op}\to{\bf BanMod}_{\mathbb X}\) as \(\tilde D(i)\coloneqq D(i)^*\)
for every \(i\in{\rm Ob}_{\bf J}\) and \(\tilde D(\phi^{\rm op})\coloneqq\textsc{Hom}(-,L^0(\mathbb X))(D(\phi)^{\rm op})\)
for every \(\phi\in{\rm Hom}_{\bf J}\). Then the limit of the diagram \(\tilde D\) is given by
\(\big(\mathscr C^*,\textsc{Hom}(-,L^0(\mathbb X))(c_\star^{\rm op})\big)\).
\end{corollary}

In particular, if \(\big(\{\mathscr M_i\}_{i\in I},\{\varphi_{ij}\}_{i\leq j}\big)\) is a direct system
in \({\bf BanMod}_{\mathbb X}\) and we define
\[
\varphi_{ij}^{\rm adj}\coloneqq\textsc{Hom}(-,L^0(\mathbb X))(\varphi_{ij}^{\rm op})
\colon\mathscr M_j^*\to\mathscr M_i^*\quad\text{ for every }i,j\in I\text{ with }i\leq j,
\]
then \(\big(\{\mathscr M_i^*\}_{i\in I},\{\varphi_{ij}^{\rm adj}\}_{i\leq j}\big)\) is an inverse system
in \({\bf BanMod}_{\mathbb X}\) whose inverse limit is given by
\begin{equation}\label{eq:inv_lim_dual}
\varprojlim\mathscr M_\star^*\cong\big(\varinjlim\mathscr M_\star\big)^*
\end{equation}
together with the morphisms \(\varphi_i^{\rm adj}\coloneqq\textsc{Hom}(-,L^0(\mathbb X))(\varphi_i^{\rm op})
\colon\big(\varinjlim\mathscr M_\star\big)^*\to\mathscr M_i^*\), where we denote by
\(\big(\varinjlim\mathscr M_\star,\{\varphi_i\}_{i\in I}\big)\) the direct limit of
\(\big(\{\mathscr M_i\}_{i\in I},\{\varphi_{ij}\}_{i\leq j}\big)\) in \({\bf BanMod}_{\mathbb X}\).
\medskip

We conclude the paper with an example. We denote by \(\ell^1\) the space of all those sequences \((a_n)_{n\in\N}\in\R^\N\)
satisfying \(\sum_{n\in\N}|a_n|<+\infty\), which is a Banach space if endowed with the componentwise operations
and the norm \(\big\|(a_n)_{n\in\N}\big\|_{\ell^1}\coloneqq\sum_{n\in\N}|a_n|\). Its dual Banach space is
the space \(\ell^\infty\) of all bounded sequences in \(\R\), endowed with the supremum norm
\(\big\|(b_n)_{n\in\N}\big\|_{\ell^\infty}\coloneqq\sup_{n\in\N}|b_n|\). It is well-known that the
space \(\ell^\infty\) does not have the Radon--Nikod\'{y}m property, thus in particular
\begin{equation}\label{eq:LB_ell_infty}
L^0(\mathbb X;\ell^1)^*\ncong L^0(\mathbb X;\ell^\infty)\quad\text{ for every finite measure space }\mathbb X,
\end{equation}
as it follows from the discussion in Remark \ref{rmk:about_RNP}.
\begin{example}\label{ex:InvIm_no_inverse}{\rm
Let \(\mathbb X=(\X,\Sigma,\mm)\) be any given finite measure space. We denote by \(\pi\colon\X\to\{p\}\)
the constant map, so that \(\pi\colon\mathbb X\to\mathbb P\) is a morphism in \({\bf Meas}_\sigma\).
We know from Example \ref{ex:LB_as_pullback} that
\begin{equation}\label{eq:LB_Banach}
\pi^*B\cong L^0(\mathbb X;B)\quad\text{ for every Banach space }B.
\end{equation}
We define the elements \((e_n)_{n\in\N}\subseteq\ell^1\) as \(e_n\coloneqq(\delta_{nk})_{k\in\N}\)
for every \(n\in\N\). Then one can readily check that \(\big(\{B_n\}_{n\in\N},\{\iota_{nm}\}_{n\leq m}\big)\)
is a direct system in \({\bf Ban}\), where \(B_n\) stands for the subspace of \(\ell^1\) generated
by \(\{v_1,\ldots,v_n\}\) and \(\iota_{nm}\colon B_n\hookrightarrow B_m\) denotes the inclusion map,
and that \(\varinjlim_\N B_\star\cong\ell^1\). Each space \(B_n\) is finite-dimensional, thus its
dual \(B'_n\) is finite-dimensional as well, and in particular it has the Radon--Nikod\'{y}m property.
Therefore, we deduce (recalling Remark \ref{rmk:about_RNP}) that
\begin{equation}\label{eq:LB_fin-dim}
L^0(\mathbb X;B_n)^*\cong L^0(\mathbb X;B'_n)\quad\text{ for every }n\in\N.
\end{equation}
Observe that \(\big(\{B'_n\}_{n\in\N},\{\iota_{nm}^{\rm adj}\}_{n\leq m}\big)\) and
\(\big(\{\pi^*B'_n\}_{n\in\N},\{\pi^*\iota_{nm}^{\rm adj}\}_{n\leq m}\big)\) are inverse systems
in \({\bf Ban}\) and in \({\bf BanMod}_{\mathbb X}\), respectively. Nonetheless, it holds that
\(\pi^*\varprojlim B'_\star\ncong\varprojlim\pi^*B'_\star\), since
\[\begin{split}
\pi^*\varprojlim B'_\star&\overset{\eqref{eq:inv_lim_dual}}\cong\pi^*\big(\varinjlim B_\star\big)'
\cong\pi^*(\ell^1)'\cong\pi^*\ell^\infty\overset{\eqref{eq:LB_Banach}}\cong L^0(\mathbb X;\ell^\infty)
\overset{\eqref{eq:LB_ell_infty}}\ncong L^0(\mathbb X;\ell^1)^*\overset{\eqref{eq:LB_Banach}}\cong(\pi^*\ell^1)^*\\
&\overset{\phantom{\eqref{eq:inv_lim_dual}}}\cong\big(\pi^*\varinjlim B_\star\big)^*
\overset{\eqref{eq:dir_lim_and_inv_im}}\cong\big(\varinjlim\pi^*B_\star\big)^*
\overset{\eqref{eq:inv_lim_dual}}\cong\varprojlim(\pi^*B_\star)^*\overset{\eqref{eq:LB_Banach}}\cong
\varprojlim L^0(\mathbb X;B_\star)^*\\
&\overset{\eqref{eq:LB_fin-dim}}\cong\varprojlim L^0(\mathbb X;B'_\star)
\overset{\eqref{eq:LB_Banach}}\cong\varprojlim\pi^*B'_\star,
\end{split}\]
which shows that the functor \({\rm InvIm}_\pi\) does not preserve inverse limits.
\fr}\end{example}
\def\cprime{$'$} \def\cprime{$'$}

\end{document}